\documentclass[11pt, reqno]{amsart}
\usepackage{graphicx, amssymb, amsmath, amsthm}
\usepackage{epsfig}
\usepackage{hyperref}
\usepackage{mathrsfs}
\numberwithin{equation}{section}
\usepackage[utf8]{inputenc}
\usepackage{tikz}
\usepackage{here}

\usepackage{pifont}
\usepackage{enumitem}

\usetikzlibrary{matrix,arrows}
\usetikzlibrary{shapes}
\usetikzlibrary{calc}
\usetikzlibrary{arrows}
\usetikzlibrary{decorations.pathreplacing,decorations.markings}
\usetikzlibrary{patterns}

\def\p{\partial}

\def\S{\Sigma}
\def\s{\mathbb{S}}
\def\I{\mathscr{I}}
\def\O{\Omega}

\def\i{\text{Int}~}
\def\n{\overset{\rightarrow}{\mathbf{n}}}

\newtheorem{theorem}{Theorem}[section]

\newtheorem{lemma}[theorem]{Lemma}

\newtheorem{corollary}[theorem]{Corollary}

\newtheorem{proposition}[theorem]{Proposition}

\theoremstyle{definition}
\newtheorem{definition}[theorem]{Definition}
\newtheorem{remark}[theorem]{Remark}

\newtheorem*{rem}{Remark}
\makeatletter
\newcommand{\Extend}[5]{\ext@arrow0099{\arrowfill@#1#2#3}{#4}{#5}}
\makeatother

\begin{document}
\title[Uniformly positive scalar curvature]{Topology of complete 3-manifolds with uniformly positive scalar curvature}

\author[Jian Wang]{Jian Wang}
\address{Current: State Key Laboratory of Mathematical Sciences, 
Academy of Mathematics and Systems Science, Chinese Academy of Sciences,
Beijing 100190, China}
\address{
Stony brook University, Department of Mathematics,100 Nicolls Road, NY-11794 Stony Brook, USA}
\email{jian.wang.4@amss.ac.cn, jian.wang.4@stonybrook.edu}

\maketitle

\begin{abstract} In this article, we  classify (non-compact) $3$-manifolds with uniformly positive scalar curvature. Precisely, we show that an orientable 3-manifold has a complete metric with  uniformly positive scalar curvature if and only if it is homeomorphic to a connected sum of spherical $3$-manifolds and some copies of $\mathbb{S}^2\times \mathbb{S}^1$. Further, we study a $3$-manifold with mean convex boundary and with uniformly positive scalar curvature. If the boundary is a disjoint union of closed surfaces, then the manifold is a  connected sum of spherical $3$-manifolds, some copies of $\mathbb{S}^1\times \mathbb{S}^2$ and some handlebodies. 
\end{abstract}
%
\section{Introduction} 

In a Riemannian manifold $(M, g)$,  the scalar curvature is defined as the sum of all sectional curvatures. The scalar curvature of $g$ is called \emph{uniformly positive}, if it is greater than a positive constant. 

When considering the topology of $3$-manifolds, particularly non-compact ones, it becomes essential to investigate those admitting (uniformly) positive scalar curvature. A well-known and fundamental question was posed by Yau  in his problem section  (See Problem 27 in \cite{Yau}):

\vspace{2mm}

\emph{How to classify $3$-manifolds admitting complete Riemannian metrics of (uniformly) positive scalar curvature up to diffeomorphisms?}

\vspace{2mm}

Extensive research has  been devoted to this question by various authors,  including Schoen-Yau's works \cite{SY3, SY1,SY},  Gromov-Lawson's \cite{GL1, GL} and others. In the case of closed $3$-manifolds, Perelman made significant contribution through  the so-called Ricci flow with surgery \cite{P1, P2, P3}. Notably, he proved that a closed and oriented $3$-manifold with positive scalar curvature is a connected sum of quotients of $\s^3$ (called spherical $3$-manifolds) and  some copies of $\s^1\times \s^2$. Bessi\`eres, Besson and Maillot extended his results to (non-compact) $3$-manifolds with uniformly positive scalar curvature and with bounded geometry (see \cite{BBM}).

\vspace{2mm}

However, the classification of  open $3$-manifolds with (uniformly) positive scalar curvature has remained unknown. The topological structure of open $3$-manifolds is much more complicated and the generalization of  Kneser's theorem \cite{Kneser} (which concerns the prime decomposition) fails  to hold for them (see \cite{Scott} and \cite{Maillot2}).

\vspace{2mm}

To facilitate our study, we introduce the concept of \emph{infinite connected sum}. 
Let $\mathscr{F}$ represent a family of $3$-manifolds.  A $3$-manifold $M$ is said to be an \emph{infinite connected sum} of members in $\mathscr{F}$, if there exists a locally finite graph $G$ satisfying the following conditions:
\begin{itemize}[leftmargin=15pt]
\item[(1)] Each vertex $v$ of $G$ corresponds to a copy $M_v$ of some manifold in $\mathscr{F}$;  

\item[(2)] For each vertex $v$, define $Y_v$ as $M_v$ with $d_v$ punctures, where $d_v$ represents the degree of $v$; 

\item[(3)] $M$ is diffeomorphic to the manifold obtained by the operation described in Figure \ref{graph1}: for each edge $e$ in $G$, connecting vertices $v, v'$, choose two $2$-spheres $S\subset \p Y_v$, $S'\subset \p Y_{v'}$, and glue $Y_v$ and $Y_{v'}$ together along $S$ and $S'$ using an orientation-reversing diffeomorphism. (See Definition \ref{inf-sum})
\end{itemize}
\begin{figure}[H]
\centering{
\def\svgwidth{\columnwidth}
{\scalebox{0.8}{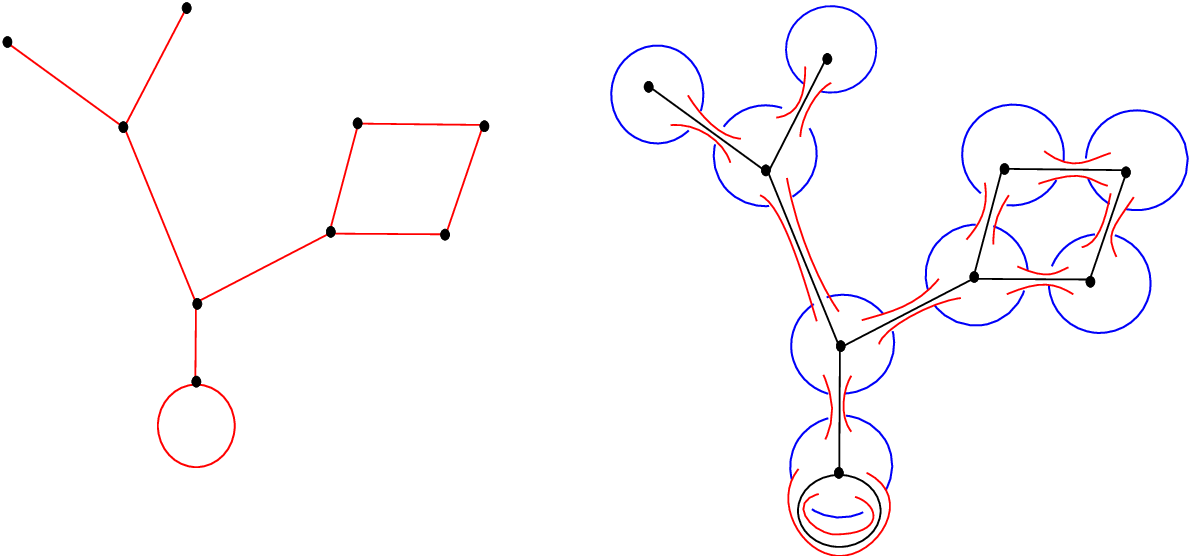}}
\caption{}
\label{graph1}
}
\end{figure}

\begin{rem}{\bf{(a)}}
For example,  $\mathbb{R}^1\times \s^2$ and  $\mathbb{R}^3$ are both the infinite connected sums of $3$-spheres (see Section 2.2).

\noindent {\bf{(b)}} When $G$ is a finite tree, it is  the usual definition of the connected sum. 

\noindent {\bf{(c)}} When $G$ is a finite graph, we could transform the graph into a tree at the expense of adding  some factors $\s^1\times \s^2$ (see Figure \ref{loops}).  Generally, for a fixed $3$-manifold $M$, the graph $G$ is not uniquely associated with $M$. It is worth mentioning that  $\s^1\times \s^2$ is a connected sum of $\s^3$ with itself.

\end{rem}

Building upon the works of \cite{GL1} and \cite{SY1}, it is well known  that  any infinite connected sum of spherical $3$-manifolds and some copies of $\mathbb{S}^1\times \mathbb{S}^2$ admits a complete metric with (uniformly) positive scalar curvature. In this paper,  we establish the converse:

\begin{theorem}\label{A} Any  complete and connected orientable $3$-manifold of uniformly positive scalar curvature is homeomorphic to an infinite connected sum of spherical $3$-manifolds and some copies of $\mathbb{S}^1\times \mathbb{S}^2$. \end{theorem}

In 2012, Besson, Bessi\`eres and Maillot  \cite{BBM} provided the first proof of Theorem \ref{A} under the assumption of bounded geometry. Recently, Gromov (see  Section 3.10.2 \cite{Gromov19}, Pages 203-204) independently presented  a concise proof for Theorem \ref{A}. 

Our  proof combines the concepts of $\mu$-bubbles and minimal surfaces. They  allow us to find  a prime decomposition for open $3$-manifolds with uniformly positive scalar curvature, which we will precisely describe later.  Coincidentally,  Gromov \cite{Gromov19} used the so-called compact exhaustion corollary (see Section 3,7.2 \cite {Gromov19}, Page 171) to get such a   decomposition. 

The second part of this paper employs  minimal surfaces to analyze the topological structure of each factor in the prime decomposition. Such a result also appeared  in the work of Chen-Chu-Zhu \cite{chen-chu-zhu} and  Gromov's note \cite{Gromov19}. Our approach could be applied to non-compact $3$-manifolds with mean convex boundaries.

We establish the following theorem:

\begin{theorem}\label{B} Let $(M, g)$ be an orientable  and connected $3$-manifold with mean convex boundaries, and with uniformly positive scalar curvature. If its boundary is a disjoint union of closed surfaces, then $M$ is homeomorphic to an infinite connected sum of spherical $3$-manifolds, handlebodies and some copies of $\mathbb{S}^1\times \mathbb{S}^2$. 
\end{theorem}
For the case of compact  $M$, this result was proven using  Ricci flow in \cite{C-L}. By   the connected sum operation in \cite {GL1, SY1}, we obtain the following corollary.  
\begin{corollary} Let $M$ be a connected  oriented $3$-manifold whose boundary is a disjoint union of closed surfaces. Then $M$ has a metric with uniformly positive scalar curvature and with mean convex boundary if and only if it is an infinite connected sum of spherical $3$-manifolds, handlebodies and some copies of $\mathbb{S}^2\times \mathbb{S}^1$. \end{corollary}
\subsection{Prime decomposition and  \texorpdfstring{$\mu$}{$\mu$}-bubbles} Introduced by Gromov in Section  $5\frac{5}{6}$ of  \cite{Gromov96}, $\mu$-bubbles are the stationary for the prescribed-mean-curvature functional. It is a breakthrough made by Chodosh and Li \cite{CL} to show that $\mu$-bubbles are closely related with the topology of the manifolds. Such a discovery also appears in many other works,  \cite{Gromov18, Gromov19, Gromov20}  \cite{Richard}\cite{zhu,zhu1}. 
\vspace{1mm}

In this article, we use $\mu$-bubbles to establish the prime decomposition for complete $3$-manifolds with uniformly positive scalar curvature. 

\vspace{1mm}

Let $(M^3, g)$ be a complete orientable $3$-manifold  with scalar curvature $\kappa\geq 1$. There are infinitely many stable $\mu$-surfaces, which serve as  topological boundaries of $\mu$-bubbles (see Lemma \ref{sep}). The uniform positivity of scalar curvature makes sure that each component of  stable $\mu$-surfaces is  a $2$-sphere (see Corollary  \ref{top}). 

These $2$-spheres from $\mu$-bubbles divide $M$ into some precompact components. Their closures  are  compact $3$-manifolds whose boundaries consist of some $2$-spheres. Then, we use the prime decomposition theorem  (see Theorem 1.5 \cite{HA}, Page 5) to obtain  a prime decomposition  for $M$. 
\begin{theorem}\label{C}Any complete, connected, orientable and open $3$-manifold with uniformly positive scalar curvature can be written as an infinite connected sum of some closed prime $3$-manifolds. \end{theorem}

 Chang, Weinberger and Yu \cite{CWY} employed K-theory to prove a similar decomposition with the assumption that the fundamental group is finitely generated. Besson-Bessi\`eres-Maillot \cite{BBM}) used Ricci flow to give the same decomposition, under a geometric condition.

\subsection{Idea of the proof for  Theorem \ref{A}} Let $(M, g)$ be a complete (non-compact) orientable $3$-manifold with scalar curvature $\kappa\geq 1$. According to Theorem \ref{C}, $M$  can be decomposed into  an infinite connected sum of some closed prime $3$-manifolds. 

Let me explain why each prime factor in the decomposition is homeomorphic to a spherical $3$-manifold or $\mathbb{S}^1\times\mathbb{S}^2$. Assume that  $\hat{M}_l$ is a closed prime $3$-manifold that appears in the decomposition of $M$. We can find  a compact $3$-manifold $M_l\subset M$ homeomorphic to $\hat{M}_l$ with an open $3$-ball removed. Furthermore, from Theorem \ref{C}, there is a non-compact $3$-manifold $\hat{M}'_l$ with \[M\cong \hat{M}_l\#\hat{M}'_l\]

We may assume that $\hat{M}_l$ is not homeomorphic to one of $\mathbb{S}^3$, $\mathbb{S}^2\times \mathbb{S}^1$ or $\mathbb{R}P^3$. Thus, $\hat{M}_l$ is irreducible and $\p M_l$ is not homotopically trivial in $M$. 

\vspace{2mm}

In $(M^3, g)$,  the area of  surfaces isotopic to $\partial M_l$ is bounded below by a positive constant. For our convenience, we assume that the area-minimizing sequence in the isotopy class containing $\partial M_l$ converges in the  sense of currents. The limit is a disjoint union of minimal $2$-spheres. We  observe that the complement of the limit has a component $M''_l$ which is homeomorphic to $\hat{M}_l$ with some punctures (see Proposition  \ref{top-replace}). Thus, $(M''_l, g|_{M''_l})$ has minimal boundary and it is $\hat{M}_l$ with some punctures. 

By the positivity of scalar curvature, we find that such a manifold has a finite fundamental group.  The Poincar\'e conjecture (see \cite{MT, cao-zhu, BBBMP}) implies that   $M''_l$ is a spherical $3$-manifold with punctures (i.e. $\hat{M}_l$ is spherical). 
\vspace{2mm}

Generally, the non-compactness of $M$ may lead that the area-minimizing sequence in the isotopy class of $\p M_l$ diverges. This phenomenon yields  a kind of boundary condition, called \emph{mixed boundary} (see Definition \ref{mixed}).  

 An orientable  $3$-manifold $(X, g)$  has \emph{mixed boundary} if the following conditions hold: 
\begin{itemize}
 \item[(a)] the boundary $\p X$ consists of  two mutually disjoint families $\{S_j\}_{j\in J'}$ and $\{S_j\}_{j\in J''}$, of 2-spheres;
 \item[(b)] for $j\in J'$,  each $S_j$ is minimal for $g$ ;
 \item[(c)] there are disjoint punctured $3$-spheres $\{\O_l\}_l$ whose boundaries contain  $\amalg_{j\in J''} S_j$;
 \item[(d)] the metric $g$ enjoys that for each $l$,
 $$d(\hat{S}_l, \p \O_l\cap \amalg_{j\in J''}S_j)>4\pi,$$
where $\hat{S}_l:=\p\O_l\cap \i~X$ is a disjoint union of some $2$-spheres.
\end{itemize}

\vspace{1mm}

To sum up, we begin with  the isotopy class of $\partial M_l$ and then use the area-minimizing sequence to construct  an open  $3$-manifold $M''_l\subset M$ (see Theorem \ref{mixed-existence}). Not only  is it homeomorphic to $\hat{M}_l$ with some punctures but also the closure of   $(M'_l, g|_{M'_l})$ is a compact $3$-manifold with mixed boundary.

These properties allow us to study the topological structure of such a manifold. 
\begin{theorem}\label{classify} Let $(M^3, g)$ be a compact $3$-manifold with scalar curvature $\kappa\geq 1$ and with mixed boundary. If $\pi_2(M)$ is generated by the spheres in $\p M$, then $\i M$ is homeomorphic to a spherical $3$-manifold with finitely many punctures. 
\end{theorem}
If $M$ is non-compact, it is homeomorphic to $\mathbb{R}^3$ with (at most) countably many disjoint $3$-balls removed. The proof of Theorem \ref{classify} follows the same approach for the  mean convex boundary case.

To show  that $M''_l$ is a spherical $3$-manifold with finitely many punctures, we use Theorem \ref{classify}, which implies that  $\hat{M}_l$ is a spherical $3$-manifold. 

\vspace{3mm}

The paper is organized as follows:
\begin{itemize} 
\item In Section 2, we  provide a review of  relevant background and introduce the infinite connected sum and boundary connected sum.
\item  Sections 3  focuses on $\mu$-bubbles and the prime decomposition. 
\item In Section 4 , we address minimal surfaces and related problems. 
\item  Sections 5 and 6 introduce the mixed boundary condition and utilize it to complete the proof of Theorem \ref{A}. 
\item In Section 7, we use a geometric version of loop lemma to prove  Theorem \ref{B}. 

\end{itemize}

We would like to thank L. Bessi\`ere, G. Besson and S. Maillot for introducing this question to me and for many helpful discussions. We are  grateful to M. Gromov, B. Lawson and B. Hanke for their interest. I am also thankful to Jintian Zhu for his helpful discussions and suggestions.  

I  am also very grateful to the referees for many helpful comments that improved the exposition of this paper. This work was  partially supported by SPP2026,  ``\emph{Geometry at infinity}".

\section{3-manifolds and infinite connected sum} In this section, we begin with several topological properties of open $3$-manifolds. Then, we introduce the so-called \emph{infinite connected sum} and give several examples.

\subsection{Topology Basis} The Poincar\'e conjecture (see \cite{MT}, \cite{BBBMP} and \cite{cao-zhu}) tells that a compact contractible 3-manifold is homeomorphic to the unit ball, $\mathbb{B}^3\subset \mathbb{R}^3$.  One useful corollary is as follows:

\begin{proposition}\label{ball}\textnormal{(See \cite{MT, BBBMP, cao-zhu} and Proposition 3.10 of \cite{HA} on Page 50)} Let $M$ be a $3$-manifold (possibly with boundary, not necessarily compact or orientable) and $\Sigma\subset M$  an embedded 2-sphere.  The sphere $\S$ is homotopically trivial  in $M$ if and only if it bounds a $3$-ball in $M$. 
\end{proposition}

If each embedded sphere in $M$ bounds a $3$-ball, such a $3$-manifold is called irreducible. In what follows, we will repeatedly use a family of irreducible $3$-manifolds with boundaries, called handlebodies. 

\begin{definition}\label{handle}(See Page 46  of \cite{rol}) A closed \emph{handlebody} is a space obtained from the closed $3$-ball $D^3$ ($0$-handle) by attaching $g$ distinct copies of $D^2\times [-1, 1]$ (1-handle) with the homeomorphisms identifying the $2g$ discs $D^2\times \{\pm1\}$ to $2g$ disjoint $2$ disks on $\p D^3$, all to be done in such a way that the resulting $3$-manifold is orientable. The integer $g$ is called the genus of the handlebody.
\end{definition}

\begin{definition}\label{bdy-con-sum}(See Page 39 of \cite{rol}) Let $M_1$ and $M_2$ be two  $n$-manifolds with boundaries. A \emph{boundary connected sum} $M_1\#_{\p}M_2$ is constructed by identifying standard balls $B^{n-1}_i\subset \p M_i$. \end{definition}

\begin{remark}\label{handle-bdy-connected} (1) A handlebody of genus $g$ is homeomorphic to a boundary connected sum of $g$ solid tori (see Page 46 of \cite{rol});\\
(2) A boundary connected sum of two handlebodies is a handlebody;\\
(3) A boundary connected sum of  a handlebody with itself is a handlebody. 
\end{remark}

\subsection{Open $3$-manifolds }  In this subsection, we assume that $M^3$ is  an open manifold which is homotopically equivalent to $\mathbb{S}^1$.

\begin{lemma}\label{cut} Let $M$ be an open $3$-manifold which is homotopically equivalent to $\mathbb{S}^1$. Then any closed embedded surface $\S\subset M$ separates $M$ into a compact part and a non-compact part. 
\end{lemma}

We will prove it in Appendix A.

\begin{corollary}\label{exhaustion} Let $M$ be an open $3$-manifold that is homotopically equivalent to $\mathbb{S}^1$. Then there is an exhaustion by compact domains $\{M_i\}_i$ of $M$ so that for each $i$, $M_i$ is compact  and $\p M_i$ is a connected closed surface. 
\end{corollary}

\begin{proof} It suffices to show that  any precompact $K\subset M$ is a subset of  some compact set $M'\subset M$, where $\p M'$ is a connected closed surface. 

\vspace{1mm}

We may assume that $K$ is connected and $\p K$ is smooth. (If not, we find a compact and connected set containing $K$ to replace it.)  The boundary $\p K$ has finitely many components, $\{\S_i\}_{i\in I}$. We use Lemma \ref{cut} to find a compact manifold $M_{\S_i}\subset M$ with $\p M_{\S_i}=\S_i$ for each $i$. 

We can conclude that there is $K$ is contained in some $M_{\S_i}$. If not, the set $\hat{K}:=K\cup_{i\in I} (\cup_{\S_i} M_{\S_i})\subset M$ is a closed manifold and $\dim (\hat{K})=\dim (M)=3$, which implies that  $M$ is not connected or $M=\hat{K}$ is compact. Since  $M$ is homotopically equivalent to $\mathbb{S}^1$,  it must be  connected and open, leading to  a contradiction.   

Some set $M_{\S_i}$ contains $K$ and its boundary is a closed and connected surface. It is the required candidate in the previous statement. 
\end{proof}

\begin{lemma}\label{curve} Let $M$ be an open $3$-manifold which is homotopically equivalent to $\mathbb{S}^1$ and $\gamma$ be  a closed curve that generates $\pi_1(M)$. Assume that  a closed surface $\S\subset M$ bounds a compact manifold $M_{\S}\subset M$. If  $\gamma$ is a subset of $M_{\S}$, there is a closed curve $\sigma\subset \S$ such that it is contractible in $M$ but non-contractible in $M\setminus \gamma$.  \end{lemma}

We prove this in Appendix A.  Corollary \ref{exhaustion} and Lemma \ref{curve} implies
\begin{corollary}\label{required-curve}Let $M$ be an open $3$-manifold which is homotopically equivalent to $\mathbb{S}^1$ and $\gamma$ be  a closed curve generating $\pi_1(M)$. For any compact set $K$ containing $\gamma$, there is a closed curve $\sigma \subset M\setminus K$ such that it is contractible in $M$ but non-contractible in $M\setminus \gamma$. 
\end{corollary}

\subsection{Infinite connected sum} We introduce the notion of an infinite connected sum. 
\begin{figure}[H]
\centering{
\def\svgwidth{\columnwidth}
{\scalebox{0.8}{\input{graph.eps_tex}}}
\caption{}
\label{graph}
}
\end{figure} 

\begin{definition}\label{inf-sum}
Let $\mathscr{F}$ be a family of $3$-manifolds. 
A $3$-manifold $M$ is an \emph{infinite connected sum} of members in $\mathscr{F}$, if there is a locally finite graph $G$ satisfying that 
\begin{itemize}[leftmargin=15pt]
\item[(1)] Each vertex $v$ of $G$ corresponds to a copy $M_v$ of some manifold in $\mathscr{F}$;  

\item[(2)] For each vertex $v$, $Y_v$ is defined as $M_v$ with $d_v$ punctures, where $d_v$ represents the degree of $v$; 

\item[(3)] $M$ is diffeomorphic to the manifold obtained by the operation described in Figure \ref{graph}: for each edge $e$ in $G$, connecting vertices $v, v'$, choose two $2$-spheres $S\subset \p Y_v$, $S'\subset \p Y_{v'}$, and glue $Y_v$ and $Y_{v'}$ together along $S$ and $S'$ using an orientation-reversing diffeomorphism. 
\end{itemize}
%
%
%

\end{definition}

When $G$ is a finite tree, it is  the usual definition of the connected sum. When $G$ is a finite graph, we could transform the graph into a tree at the expense of adding  some factor $\s^1\times \s^2$ (see Figure \ref{loops}).

\begin{rem} (a) For any closed manifold $M$, $M\# (\mathbb{S}^1\times \mathbb{S}^2)$ can be written as the connected sum of $M$ with itself.

(b) The associated graph in Definition \ref{inf-sum} is not unique. For example, the manifold $M_1$ constructing from the graph $G_1$ is homeomorphic to $M_2$ from $G_2$, where $M_{v_0}$ is $\mathbb{S}_1\times \mathbb{S}_2$ (see Figure 3). 
\end{rem}

\begin{figure}[H]
\centering{
\def\svgwidth{\columnwidth}
{\scalebox{0.7}{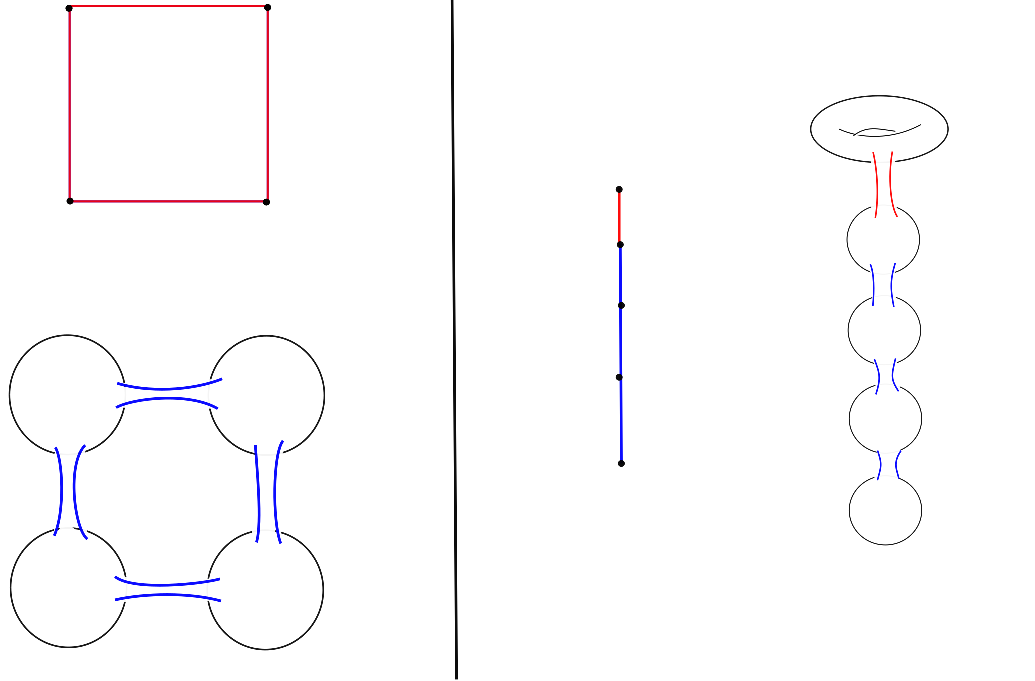}}
\caption{}
\label{loops}
}
\end{figure}

Some non-compact $3$-manifold  can be written as an infinite connected sum of some closed $3$-manifolds, for example, $\mathbb{R}^3$ is an infinite connected sum of  some copies of $\s^3$(see Figure 4).

\begin{figure}[H]
\centering{
\def\svgwidth{\columnwidth}
{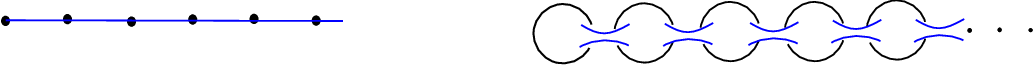}
\caption{}
\label{R3}
}
\end{figure}

\section{$\mu$-bubble and Prime decomposition}

Classically, any closed and oriented $3$-manifold can be written as a connected sum of some prime $3$-manifolds (see \cite{Kneser}). Our focus  will be  on the prime decomposition for complete $3$-manifolds with uniformly positive scalar curvature.

My main idea  is to find a family of $2$-spheres which cut the complete $3$-manifold into some pre-compact $3$-manifolds. Then, we apply the prime decomposition to each component and  get a required decomposition for the complete $3$-manifold. However, the decomposition is not unique. 

\vspace{3mm}

The existence of such a family of $2$-spheres comes from $\mu$-bubble theory.

\subsection{$\mu$-bubbles}   In this section, we recall the general existence and the stability theorem for $\mu$-bubbles. 

\begin{definition} Let $(M, g)$ be a connected, compact  manifold with boundary and $\p M:=\p_-M\amalg \p_+ M$  be a choice of labeling the components of $\p M$ so that neither of the sets are empty. \\
(1) The \emph{reduced boundary} $\partial^* \Omega $ of a Caccioppoli set $\Omega\subset M$ is equal to $\p \Omega \setminus (\partial_-M\amalg \partial_+ M)$; \\
(2) Let $\Omega_0\subset M$ be a Caccioppoli set with $\partial_+ M\subset \Omega_0$. Each smooth function $h$ on $\i M$ is equipped with a functional $\mathcal{A}_h$: 
\begin{equation}\label{def-bubble}\mathcal{A}_h(\Omega)=\mathcal{H}^2(\p^*\Omega)-\int_M(\chi_{\O}-\chi_{\O_0})h~d\mathcal{H}^3, \end{equation} where $\O$ is any Caccioppoli set  of  $M$ with $\O\Delta\O_0\Subset \text{Int}~M$
and $\chi_\O$ is the characteristic function of region $\O$. 
\end{definition}
Note that if $h=0$  and $\O$ is a critical point of $\mathcal{A}_h$, then $\p^*\O$ is minimal. 

\vspace{2mm}

The existence and regularity of a minimizer of $\mathcal{A}_h$ among all Caccioppoli sets is claimed by Gromov (see Section 5.1 of \cite{Gromov19} ) and rigorously carried out by Zhu (see Proposition 2.1 of  \cite{zhu} on Page 4). 

\begin{proposition}\label{exist}\textnormal{(See Proposition 2.1  \cite{zhu}, Page 4)} If the function $h$ satisfies 
\begin{equation}
\lim_{x\rightarrow \p_-M} h(x)=-\infty ~~\text{and}~~\lim_{x\rightarrow \p_+M} h(x)=+\infty, 
\end{equation} then there exists a smooth minimizer $\O$ for $\mathcal{A}_h$ such that $\O\Delta \O_0$ is compactly contained in $\text{Int}~M$.
\end{proposition}

\begin{proof}Let $\mathcal{C}$ be the collection of Caccioppoli sets $\O$ in $\i~ M$ with  $\O\Delta \O_0\Subset \i~M$ and define that 
\[I:=\inf\{\mathcal{A}_h(\O)~| ~\O\in \mathcal{C}\}.\]
We first prove $I>-\infty$ and then show that a minimizing sequence converges to the required candidate. 

\vspace{2mm}

We now show that $I>-\infty$. For any $\tau>0$, let $\O^\tau_{\pm}$ be the distance-$\tau$ neighborhood of $\p_{\pm}M$. They are two families $\{\Omega^\tau_+\}$ and $\{M\setminus \O^\tau_-\}$ of Caccioppoli sets with $\O^\tau_+ \Delta \O_0\Subset M$ and $(M\setminus \O^\tau_-) \Delta \O_0\Subset M$. Their reduced boundaries $S^\tau_\pm:= \partial^* \O_\pm^\tau$ are smooth equidistant hypersurfaces to $\p_{\pm} M$ for sufficiently small $\tau$. 

Note that $\Omega^\tau_\pm$ is foliated by smooth surfaces $\{S^{\rho}_{\pm}\}_{\rho\in [0, \tau ]}$, where $S^\rho_\pm=\p^* \O^\rho_\pm$. The unit normal vector $\eta$ of $\{S^{\rho}_{\pm}\}_{\rho\in [0, \tau ]}$, pointing out $M$ along $\p_\pm M$,  is a smooth vector field on $\Omega^\tau_{\pm}$.  Thus, $|\text{div}(\eta)|$ is finite on $\Omega^\tau_\pm$. Choosing $\tau$ small enough, we find 
\begin{equation}\label{div+bdy1}
h>\text{div}(\eta) ~\text{in}~\O^\tau_{+} \end{equation}
\begin{equation}\label{div+bdy2}  h<-\text{div}(\eta) ~\text{in}~ \O^\tau_{-}. 
\end{equation}

Observe that for any Caccioppoli set $\O$ with $\partial_+M \subset \O$,  we have that  
$\partial^* \O \Delta \partial^*(\O\cup \O^\tau_+)$ is a disjoint union of $\partial \O^{\tau}_+\setminus \O$ and $\partial \O\setminus \O^\tau_+$. This observation and \eqref{div+bdy1} allow for computing the difference between $\mathcal{A}_h(\O^\tau_+\cup \O)$ and $\mathcal{A}_h(\O)$. 
\begin{equation}\label{pre-comp}
\begin{split}
\mathcal{A}_h(\O\cup\O^{\tau}_+)-\mathcal{A}_h(\O)&=\int_{\p^*(\O\cup\O^\tau_+)}d\mathcal{H}^2-\int_{\p^*\O}d\mathcal{H}^2-\int_{\O^\tau_+\setminus \O}h ~d\mathcal{H}^3\\
&=\int_{\p^*\O^\tau_+\setminus \O}d\mathcal{H}^2-\int_{\p^*\O\cap \O^\tau_{+}} d\mathcal{H}^2-\int_{\O^\tau_+\setminus \O}h~d\mathcal{H}^3, \\
& <\int_{\p^*\O^\tau_+\setminus \O}d\mathcal{H}^2-\int_{\p^*\O\cap \O^\tau_{+}} d\mathcal{H}^2-\int_{\O^\tau_+\setminus \O}\text{div}(\eta)~d\mathcal{H}^3\end{split}
\end{equation}
 Moreover, $\partial \O^\tau_+\setminus \O$ is a disjoint union of $\partial^* \O^\tau_+\setminus \O$ and $\partial^* \O\setminus \O^\tau_+$. Basic computations yield that 
\begin{equation}\label{div-comp}
\begin{split}
\int_{\O^{\tau}_+\setminus \O}\text{div}(\eta)&=\int_{\p^*\O^\tau_+\setminus \O} \eta \cdot \nu ~d \mathcal{H}^2-\int_{\p*\O\cap \O^\tau_+} \eta\cdot \nu ~d\mathcal{H}^2\\
&\geq \int_{\p^*\O^\tau_+\setminus \O} d\mathcal{H}^2-\int_{\p^*\O\cap \O^\tau_+} d\mathcal{H}^2, 
\end{split}
\end{equation}
where $\nu$ is the unit normal vector and $\eta\cdot \nu=1$ on $\p^{*}\O^{\tau}_+\setminus \O$. Combining \eqref{pre-comp} with \eqref{div-comp}, we have that 
 \begin{equation}\label{rep1}\mathcal{A}_h(\O\cup\O^{\tau}_+)-\mathcal{A}_h(\O)\leq 0\end{equation}

Using  \eqref{div+bdy2} and the divergence theorem, we  analogously calculate the difference between $\mathcal{A}_h(\O)$ and $\mathcal{A}_h(\O\setminus \O^\tau_-)$ and get that \begin{equation}\label{rep2}\mathcal{A}_h(\O\setminus \O^\tau_-)<\mathcal{A}_h(\O).  
\end{equation} 
We note that the proof of \eqref{rep2} only depends on \eqref{div+bdy2} and the divergence theorem. Using the set difference instead of the set union does not cause any issue. 

From \eqref{rep1} and \eqref{rep2}, we find that 
$
\mathcal{A}_h(\O)>\mathcal{A}_h((\O\cup\O^\tau_+)>\mathcal{A}_h((\O\cup\O^\tau_+)\setminus \O^\tau_-)
$. Then,  observe that  $h$ is bounded by some constant $C$ on $M\setminus (\O^{\tau}_+\cup\O^\tau_-)$, which implies that  
$$\mathcal{A}((\O\cup\O^\tau_+)\setminus \O^\tau_-)> -\int_{M\setminus (\O^{\tau}_+\cup\O^\tau_-)} |h| d\mathcal{H}^3>-C \mathcal{H}^3(M).$$
Thus, $I>-\infty$. 

\vspace{2mm}

Take a sequence $\{\O_k\}\subset \mathcal{C}$ with $\mathcal{A}_h(\O_k)\rightarrow I$ as $k\rightarrow \infty$. Note that by using \eqref{rep1} and \eqref{rep2}, we find $\{\mathcal{A}_h((\O_k\cup \O^\tau_+)\setminus \O^\tau_-)\}_k$ is also an minimizing sequence for $\tau$ sufficiently small. We may assume that 
\begin{equation}
\O_k\Delta \O_0 \subset M\setminus (\O^\tau_+\cup\O^\tau_-). 
\end{equation} Choosing  $k$ large enough, $\mathcal{A}_h(\O_k)$ is bounded by $I+1$. Combining with the bound of $h$ on $M\setminus (\O^{\tau}_+\cup\O^\tau_-)$, we get  that 
$$\mathcal{H}^2(\p^*\O_k)\leq I+1 +C \mathcal{H}^3(M)$$
By BV-compactness, taking a subsequence, the sets $\O_k$ converge to a Caccioppoli set $\hat{\O}$. It follows that $\hat{\O}$ is a minimizer of $\mathcal{A}_h$ and thus has smooth boundary from the standard regularity theory \cite{Italo}. \end{proof}

\subsection{First and Second variations}
\begin{lemma}\label{first-var} If $\O_t$ is a smooth 1-parameter family of regions with $\O_t|_{t=0}=\O$ and normal speed $\phi$ at $t=0$, then 
\begin{equation}
\frac{d}{dt}\mathcal{A}_h(\O_t)\Big|_{t=0}=\int_{\Sigma_0} (H-h)\phi ~d\mathcal{H}^2
\end{equation}where $\Sigma_0:=\p^*\O$ and  $H$ is the mean curvature of $\Sigma_0$. In particular, a critical point $\O$ of $\mathcal{A}_h$ satisfies 
\begin{equation}
H=h
\end{equation} along $\S_0$. 

\begin{proof} The reduced boundary $\Sigma_t:=\partial^* \Omega_t$ gives a variation  $F: (-\epsilon, \epsilon)\times \Sigma_0\rightarrow M$ with $F(t, \Sigma_0)=\Sigma_t$. The variational field is $F_t(0, x)$ and its normal component $F^N_t(0, x)=\phi\n_{\Sigma_0}(x)$. The first variation of area functional shows 
\begin{equation}\label{1v1t}
\begin{split}
\frac{d}{dt} \int_{\Sigma_t} d\mathcal{H}^2\Big|_{t=0}&=\int_{\Sigma_0} H\phi ~d\mathcal{H}^2+\int_{\Sigma_0} \text{div}_{\Sigma_0} (F^T_t(0, x))~d\mathcal{H}^2,\\
&=\int_{\Sigma_0} H\phi~d\mathcal{H}^2,
\end{split} \end{equation}where $\text{div}_{\Sigma_0}$ is the divergence operation of the induced metric on $\Sigma_0$ and $F^T_t(0, x)$ is the tangential component of $F_t(0, x)$. The second equation comes from Stokes' theorem. The variation of the second term in \eqref{def-bubble} is 
 \begin{equation}\label{1v2t}\frac{d}{dt}\int_{\O_t}h~d\mathcal{H}^3\Big|_{t=0}=\int_{\Sigma_0}h\phi~d\mathcal{H}^2. \end{equation}Thus, we have that $\frac{d}{dt}\mathcal{A}_h(\O_t)\Big|_{t=0}=\int_{\Sigma_0} (H-h)\phi ~d\mathcal{H}^2$\end{proof}

\end{lemma}

\begin{lemma}\label{second-var} Let $\O$ be a minimizer of $\mathcal{A}_h$ with the reduced boundary $\S_0:=\p^*\O$. Then,  for
 any smooth function $\phi$ on $\Sigma_0$, we have that 
\begin{equation}\label{sec}\int_{\Sigma_0}|\nabla_{\Sigma_0} \phi |^2+ K_{\Sigma_0} \phi^2-(\frac{1}{2}|A|^2+\kappa_M+\frac{h^2}{2}+\langle\nabla_M h, \n_{\Sigma_0}\rangle)\phi^2~ d\mathcal{H}^2\geq 0
\end{equation}
where $|A|^2$ is the squared norm of the second fundamental form of $\Sigma_0$, $\kappa_M$ is the scalar curvature of $(M, g)$ and  $K_{\S_0}$ is the sectional curvature of the induced metric on $\S_0$.\end{lemma}

\begin{proof} Let $\Omega_t$ and $F$ be assumed as in the proof of Lemma \ref{first-var}. For our convenience, we may assume that $F$ is a normal variation on $\Sigma_0$:  \[F_t (0, x)=\phi\n_{\Sigma_0}. \]  The second variational formula of the area functional states that 
\begin{equation*}\begin{split}
\frac{d^2}{dt^2}\int_{\Sigma_t} d\mathcal{H}^2\Big{|}_{t=0}&=\int_{\Sigma_0} |\nabla_{\Sigma_0} \phi|^2-(\text{Ric}(\n_{\Sigma_0})+|A|^2)\phi^2+H^2\phi^2 d\mathcal{H}^2 +\int_{\Sigma_0} \text{div}_{\Sigma_0} (F_{tt}(0, x)) d\mathcal{H}^2,\\
&=\int_{\Sigma_0} |\nabla_{\Sigma_0} \phi|^2-(\text{Ric}(\n_{\Sigma_0})+|A|^2)\phi^2+H^2\phi^2 d\mathcal{H}^2.
\end{split}
\end{equation*}
where $H$ is the mean curvature of $\Sigma_0$. The last equation follows from Stokes' theorem, since  that $F_{tt}(0, x)$ belongs to $T_x  \Sigma_0$, where $x\in \Sigma_0$. Differentiating Equation \eqref{1v2t} implies that 

\begin{equation}\label{2v2t}
\frac{d^2}{dt^2}\int_{\Omega_t}h~ d \mathcal{H}^3\Big{|}_{t=0}=\int_{\Sigma_0} \langle\nabla_M h, \n_{\Sigma_0}\rangle\phi^2 +h H\phi^2 ~d\mathcal{H}^2.\end{equation}
Observe that the minimizing property of $\mathcal{A}_h(\O)$ implies that the mean curvature $H$ of ${\Sigma_0}$ is $h|_{\Sigma_0}$ (by  Lemma \ref{first-var}) and  $\frac{d^2}{dt^2}\mathcal{A}_h(\Omega_t)|_{t=0}$ is non-negative. This  observation produces  that 
\begin{equation*}\begin{split}
\frac{d^2}{dt^2} \mathcal{A}_h(\Omega_t)\big{|}_{t=0}&=\int_{\Sigma_0}|\nabla_{\Sigma_0} \phi|^2 -(\text{Ric}(\n_{\Sigma_0})+|A|^2)\phi^2 - \langle\nabla_M h, \n_{\Sigma_0}\rangle\phi^2 ~d\mathcal{H}^2\\
&\geq 0.
\end{split}
\end{equation*}
W use the key observation given by Schoen and Yau  (see Page 165, \cite{SY1}) 
\begin{equation}\label{SY-OB}\text{Ric}_M(\n_{\S_0})=\kappa_M-K_{\S_0}+\frac{H^2}{2}-\frac{1}{2} |A|^2, \end{equation}
to plug  into the above inequality
\[\int_{\Sigma_0}|\nabla_{\Sigma_0} \phi |^2+ K_{\Sigma_0} \phi^2-(\frac{1}{2}|A|^2+\kappa_M+\frac{h^2}{2}+\langle\nabla_M h, \n_{\Sigma_0}\rangle)\phi^2~ d\mathcal{H}^2\geq 0
\]\end{proof}

\subsection{Free boundary $\mu$-bubble} Consider a Riemannian $3$-manifold  $(M, g)$ with boundary $\p M:=\p_+M\amalg \p_-M\amalg\p_0 M$, where $\p_{\pm}M$ are non-empty smooth submanifolds of $M$.

In what follows, we  assume that $\p_0 M$ is minimal for $g$.  This prevents tangential contact between  free boundary $\mu$-bubble and $\p_0 M$ in the usual way that free boundary minimal surface does not make tangential contact with mean convex components of the boundary. 

\begin{definition}
(1) The \emph{reduced boundary} $\partial^* \Omega $ of a Caccioppoli set $\Omega\subset M$ is equal to $\p \Omega \setminus (\partial_-M\amalg \partial_+ M)$; \\
(2) Let $\Omega_0\subset M$ be a Caccioppoli set with $\partial_+ M\subset \Omega_0$. Each smooth function $h$ on $\i M\cup \partial_0 M$ is equipped with a free boundary functional $\mathcal{A}^*_h$: 
\begin{equation}\mathcal{A}^*_h(\Omega)=\mathcal{H}^2(\p^*\Omega)-\int_M(\chi_{\O}-\chi_{\O_0})h~d\mathcal{H}^3, \end{equation} where $\O$ is any Caccioppoli set  of  $M$ with $\O\Delta\O_0\Subset \text{Int}~M\amalg \p_0 M$. 
\end{definition}

We note that if $h=0$, the reduced boundary $\partial^* \Omega$ of a minimizer $\Omega$ for $\mathcal{A}^*_h$ is  a free boundary minimal surface. By similar arguments to Proposition \ref{exist}, we can conclude that 

\begin{proposition}\label{var-free} Let $h$ be a smooth function on $\i M\cup \p_0 M$. If it satisfies that 
\begin{equation*}\lim_{x\rightarrow \p_+ M} h(x)=+ \infty;\\
 \lim_{x\rightarrow \p_- M} h(x)=- \infty, \end{equation*}
 then there exists $\O$ with the smooth reduced boundary $\p^* \O$ minimizing $\mathcal{A}^*_h$ among such regions. 
 \end{proposition}
 
%
%

 \begin{lemma}\label{free-var1} Let $M$ and $\mathcal{A}^*_h$ be defined as above and $\Omega$ a critical point of $\mathcal{A}^*_h$. Then  the mean curvature of the reduced boundary $\S: =\p^* \O$ of $\O$ is  
   \[H=h|_\Sigma,\] 
   and   $\Sigma$ meets $\p_0M$ orthogonally. \end{lemma}
 
 \begin{proof}Let $\Omega_t$ be a smooth 1-parameter family of regions with $\Omega_t \Big{|}_{t=0}=\Omega$ and $F: (-\epsilon, \epsilon)\times \p^*\O\rightarrow M$ be a variation of $\Sigma$, with $F(t, \p^*\O)=\p^*\O_t$.  We remark that $\partial \Sigma$ may be a subset of $\p_0 M$. For our convenience, we assume 
 \begin{equation*}
\begin{cases}
F_t(0, x)\in T_x(\p_0 M)\quad \quad  \text{ for } x\in \partial \Sigma\\

F^N_t(0, x)=\phi(x)\n_\S(x) \quad \text{on}\quad x\in \Sigma, 
\end{cases}
\end{equation*}where $\phi$ is a smooth function on $\Sigma$ and $F^T_t(0, x)$ and $F^N_t(0,x)$ are the tangential and normal component of $F_t(0, x)$ respectively. 
 The first variational formula for the area functional gives that 
 \begin{equation*}
\begin{split}
\frac{d}{dt} \int_{\Sigma_t} d\mathcal{H}^2\Big|_{t=0}&=\int_{\Sigma} H\phi ~d\mathcal{H}^2+\int_{\Sigma} \text{div}_{\Sigma} (F^T_t(0, x))~d\mathcal{H}^2,\\
&=\int_{\Sigma} H\phi~d\mathcal{H}^2+ \int_{\p \Sigma}\langle F^{T}_t(0, x), \overset{\rightarrow}{\alpha}\rangle d\mathcal{H}^1 \quad \quad \quad \text{By Stokes's theorem}\\
&=\int_{\Sigma} H\phi~d\mathcal{H}^2+ \int_{\p \Sigma}\langle F_t(0,x), \overset{\rightarrow}{\alpha}\rangle\phi d\mathcal{H}^1\quad\quad \quad \text{since}\langle F^N_t(0,x), \overset{\rightarrow}{\alpha}\rangle=0.
\end{split} 
\end{equation*}
where $\overset{\rightarrow}{\alpha}$  is  the outward normal vector of $\p\S$ in $\S$.
Combining with  \eqref{1v2t} , we can have that \[\frac{d}{dt}\mathcal{A}(\O_t)\Big|_{t=0}=\int_{\Sigma} (H-h)\phi ~d\mathcal{H}^2+\int_{\p \Sigma}\langle F_t(0,x), \overset{\rightarrow}{\alpha}\rangle\phi d\mathcal{H}^1.\]
Since  $\frac{d}{dt}\mathcal{A}(\O_t)\Big|_{t=0}=0$ for any $\phi$, we observe that the  mean curvature $H$ of  $\Sigma$ is $h|_{\Sigma}$ and  $\langle F_t(0,x), \overset{\rightarrow}{\alpha}\rangle=0$ on $\p \S$ for any variation $F$. We note that $F_t(0, x)$ can be chosen as any vector in $T_x(\p_0 M)$, which implies that  $\Sigma$ meets $\p_0M$ orthogonally.  
  \end{proof}
 \begin{lemma}\label{free-var2}Let  $\O$ be a minimizer of $\mathcal{A}^*_h$ and $\S:=\p^*\O$ be its reduced boundary. Then,   for any $\phi\in C^1(\S)$, one has that 
\begin{equation}\label{2var-free}\begin{split}
\int_{\Sigma}&(|\nabla_{\Sigma}\phi |^2+ K_\S \phi^2)~d\mathcal{H}^2+\int_{\p \S}k_g\phi^2~d\mathcal{H}^1
 -\frac{1}{2}\int_\S |A|^2\phi^2 d\mathcal{H}^2\\
&-\int_\S(\frac{h^2}{2} +\langle\nabla_M h, \n_{\Sigma}\rangle+\kappa_M)\phi^2~ d\mathcal{H}^2 \geq 0,
\end{split}
\end{equation}where $k_g$ is the geodesic curvature of $\p \S$ as a curve in  $\S$. 
\end{lemma} 
\begin{proof} Let $F$ be a variation defined in the proof of Lemma \ref{free-var1} with $F_t(x, 0)=\phi(x)\n_{\S}(x)$ on $\S$. Basic computations yield that 
\begin{equation*}\begin{split}
\frac{d^2}{dt^2}\int_{\Sigma_t} d\mathcal{H}^2\Big{|}_{t=0}&=\int_{\Sigma} \{|\nabla_{\Sigma} \phi|^2-(\text{Ric}(\n_\Sigma)+|A|^2)\phi^2+H^2\phi^2 d\mathcal{H}^2 +\int_{\Sigma} \text{div}_\Sigma (F_{tt}(0, x)) d\mathcal{H}^2,\\
&=\int_{\Sigma} |\nabla_{\Sigma} \phi|^2-(\text{Ric}(\n_\Sigma)+|A|^2)\phi^2+H^2\phi^2 d\mathcal{H}^2 + \int_{\p \S}\langle F_{tt}(0,x), \overset{\rightarrow}{\alpha} \rangle  d\mathcal{H}^1\\
&=\int_{\Sigma} |\nabla_{\Sigma} \phi|^2-(\text{Ric}(\n_\Sigma)+|A|^2)\phi^2+H^2\phi^2 d\mathcal{H}^2 + \int_{\p \S}\langle \nabla_{\n_\S}\n_\S, \overset{\rightarrow}{\alpha} \rangle \phi^2 d\mathcal{H}^1
\end{split}
\end{equation*}
 where $\overset{\rightarrow}{\alpha}$ is the unit outward normal vector of $\p \Sigma$ in $\Sigma$.

 Let $\overset{\rightarrow}{\beta}$ be the unit tangential vector of $\p \S$. We observe that by  Lemma \ref{free-var1}, we find that for $x\in \p \S$, $\overset{\rightarrow}{\alpha}(x)$ is the unit normal vector of $\p_0 M$  and  $T_x(\p_0 M)$ can be spanned by $\overset{\rightarrow}{\beta}(x)$ and $\n_\S$. We may write that along $\p \S$
 \[H(x)=-\langle \overset{\rightarrow}{\alpha}, \nabla_{\overset{\rightarrow}{\beta}}\overset{\rightarrow}{\beta}\rangle -\langle \overset{\rightarrow}{\alpha}, \nabla_{\n_\S} {\n_\S}\rangle=0\text{ and }
 k_g=-\langle \nabla_{\overset{\rightarrow}{\beta}}\overset{\rightarrow}{\beta}, \overset{\rightarrow}{\alpha}\rangle \]
Thus,  the boundary term of the above second variation  can be rewritten: 
\begin{equation}\label{bdy-geo}\int_{\p \S}\langle \nabla_{\n_\S}\n_\S, \overset{\rightarrow}{\alpha} \rangle \phi^2 d\mathcal{H}^1=\int_{\p \S}k_g \phi^2 d\mathcal{H}^1. \end{equation}

Since $\O$ is a minimizer of $\mathcal{A}^*_h$, we have that  $H=h$ on $\Sigma$ and $\frac{d^2}{dt^2}\mathcal{A}^*_{h}(\Omega_t)|_{t=0}\geq 0$. Combining with \eqref{bdy-geo} \eqref{2v2t} and \eqref{SY-OB}, we get that 
\[\begin{split}
\frac{d^2}{dt^2} \mathcal{A}^*_h(\Omega_t)\big{|}_{t=0}=&\int_{\Sigma}|\nabla_{\Sigma} \phi|^2 -(\text{Ric}(\n_\Sigma)+|A|^2)\phi^2 - \langle\nabla_M h, \n_{\Sigma}\rangle\phi^2 ~d\mathcal{H}^2\\
&+\int_{\p\Sigma} k_g \phi^2~d\mathcal{H}^1 \quad \quad \quad\quad \quad \quad \quad \quad \quad \quad \quad \quad \text{ from } \eqref{bdy-geo} \text{ and } \eqref{2v2t}\\
&=\int_{\Sigma}(|\nabla_{\Sigma} \phi|^2 +K_\S \phi^2)~d \mathcal{H}^2+\int_{\p \S} k_g \phi^2 ~d \mathcal{H}^1- \int_{\S} \frac{1}{2} |A|^2\phi^2~d \mathcal{H}^2 \\
&-\int_\S(\kappa_M+h^2 + \langle\nabla_M h, \n_{\Sigma}\rangle)\phi^2 ~d\mathcal{H}^2\quad \quad \quad \text{ by } \eqref{SY-OB}\\
&\geq 0.
\end{split}
\]
\end{proof}
\subsection{Topology of stable $\mu$-surfaces and the Separation Lemma}
Let $(M, g)$ be a Riemannian manifold whose boundary $\p M:=\p_+M \cup \p_-M\cup \p_0M $ is a choice of labeling the components of $\p M$ such that (a) $\p_\pm M$ are non-empty; (b) if $\p_0 M$ is non-empty, it is minimal for $g$. 

\begin{remark}\label{choice-h} If $d(\p_+ M, \p_- M)> 4\pi$, then there is a smooth function $h$ on $\i M\cup \p_0 M$ satisfying that
\[\lim_{x\rightarrow \p_\pm M}h(x)=\pm \infty \text{ and } 1+\frac{1}{2}h^2- |\nabla h|>0. \] 
The construction is as follows: We begin with  two  smooth  functions $\tau_1, \tau_2 $:
\[\begin{cases}
\tau_1(0)=0  \quad \text{and} \quad\tau_1(t)>0 \text{ for } t>0,\\
\tau_1 \text{ is non-decreasing }\\
\tau_1(t)=4\pi \text{ for } t>4\pi+\epsilon\\
|\text{Lip} (\tau_1)|<1
\end{cases}\begin{cases}
\tau_2(0)=4\pi \text{ and } \tau_2(t)<4\pi \text{ for } t>0\\
\tau_2 \text{ is non-increasing}\\
\tau_2(t)=0 \text{ for } t\geq 4\pi+\epsilon\\
|\text{Lip} (\tau_2)|<1 
\end{cases}
\] where $\epsilon:=d(\p_+M, \p_-M)-4\pi$. Consider the functions $d_+(x):=d(x, \p_+ M)$, $d_-(x):=d(x, \p_-M)$ and define $$\rho(x)=(1-\epsilon)\tau_1(d_-(x))+\epsilon \tau_2(d_+(x)). $$
Then, we have that (a)$\rho|_{\p_-M}=0$ and $\rho|_{\p_+M}=4\pi$,  (b) $0<\rho(x)<4\pi$ for $x\in \i M\cup\p_0M$ and $|\text{Lip}(\rho)|<1$.  We choose $h$ as follows: 
   \begin{equation}\label{def-h}h:=\tan (\frac{\rho(x)-2\pi}{4}),\end{equation}
   Basic computations easily check the properties of $h$. 
\end{remark}

\begin{corollary}\label{top1}Let $(M^3, g)$ be an oriented and compact manifold with 
\begin{itemize}
\item[(a)]  $\p M=\p_-M\cup \p_+ M$ and with $d(\p_+ M, \p_- M)>4\pi$
\item[(b)]   the scalar curvature $\kappa_M\geq 1$
\end{itemize}
and 
$h$ be a function defined in Remark \ref{choice-h}. If $\O$ is a minimizer of $\mathcal{A}_h$, then each component of the reduced boundary $\p^* \O$ is a 2-sphere. 
\end{corollary}

\begin{proof} The choice of $h$ (see Remark \ref{choice-h}) and the lower bound of $\kappa_M$ show  \[\kappa_M+\frac{h^2}{2}-|\nabla h|>0 \]
For our convenience, we may assume that $\S:=\p^*\O$  is connected. We use \eqref{sec} to have that for any smooth function $\phi$ on $\Sigma$, 
\[\int_{\Sigma}|\nabla_{\Sigma} \phi |^2+ K_\Sigma \phi^2\geq \int_M \frac{1}{2}|A|^2\phi^2d\mathcal{H}^2+\int_M(\kappa_M+\frac{h^2}{2}+\langle\nabla_M h, \n_{\Sigma}\rangle)\phi^2~ d\mathcal{H}^2> 0\]
Choosing  $\phi=1$ implies that  $\chi(\Sigma)=\int_\S K_{\S}d\mathcal{H}^2>0$ (i.e  $\S$ is a $2$-sphere or a $\mathbb{R}P^2$). Since $M$ is oriented and the normal bundle of $\Sigma$ is trivial, we can conclude that $\Sigma$ is a $2$-sphere. \end{proof}

\begin{corollary} \label{top}Let $(M^3, g)$ be an oriented and compact manifold satisfying 
 \begin{itemize}
\item[(1)]  $\p M=\p_-M\cup \p_+ M\cup\p_0 M$ and $\p_0 M$ is minimal for $g$;
\item[(2)]  $d(\p_+ M, \p_- M)>4\pi$;
\item[(3)]   the scalar curvature $\kappa_M\geq 1$
\end{itemize}
and 
$h$ be a function defined in Remark \ref{choice-h}.  If $\O$ is  a minimizer of $\mathcal{A}^*_h$, then each component of $\p^*\O$ is a $2$-sphere or a disc. 
\end{corollary}
\begin{proof}  For simplifying the proof, we may assume that $\S:=\p^* \O$ is a connected surface with boundary (Otherwise, we use the argument in the proof of Corollary \ref{top1} to show $\S$ is a $2$-sphere.)
We  use \eqref{2var-free} to obtain that for any smooth function $\phi$ on $\Sigma$
\[
\int_{\Sigma}|\nabla_{\Sigma}\phi |^2+ K_\S \phi^2~d\mathcal{H}^2+\int_{\p \S}k_g\phi^2~d\mathcal{H}^1
 \geq \frac{1}{2}\int_\S |A|^2 \phi^2d\mathcal{H}^2\\
+\int_\S(\frac{h^2}{2} +\langle\nabla_M h, \n_{\Sigma}\rangle+\kappa_M)\phi^2~ d\mathcal{H}^2 >0. 
\]
Choosing $\phi=1$ shows  $\chi(\Sigma)=\int_\S K_\S ~d\mathcal{H}^2+\int_{\p \S}k_g >0$. Combining with the facts that $M$ is oriented and $\Sigma$ has a trivial normal bundle, we  can conclude that $\S$ is a disc. 
\end{proof}

We use  Corollary \ref{top} to have the following separation lemma. 

\begin{lemma}\label{sep}(The separation lemma) Let $(M, g)$ be a compact and orientable $3$-manifold satisfying that 
\begin{itemize}
\item[(a)] $\p M=\p_+ M\amalg \p_-M\amalg \p_0 M$; 
\item[(b)] $d(\p_- M, \p_+ M)> 4\pi$;
\item[(c)] $\p_0 M$ is minimal for $g$ and a disjoint union of $2$-spheres. 
\end{itemize}
Then, if the scalar curvature $\kappa_M\geq 1$,  there is  a finite collection $\{S_i\}_i$ of embedded $2$-spheres in $\i~M$ which satisfies the following conditions
\begin{itemize}
\item $\{S_i\}_i$ are disjoint;
\item  $M\setminus \amalg_iS_i$ has no component intersecting both   $\p_-M$  and $\p_+ M$. 
\end{itemize}
\end{lemma}
\begin{proof} Choose the function $h$ as in Remark \ref{choice-h}  and $\O_0=\{x~|~ d(x, \p_-M)\leq 2\pi\}\subset M$. By Proposition \ref{var-free}, there is a minimizer $\O$ for the functional $\mathcal{A}^*_h$ with $\O_0 \Delta \O\Subset \i M\cup \p_0 M$. We note that the complement of the reduced boundary $\p^*\O:=\amalg_i \Sigma_i$ has no component intersecting both $\p_+M $ and $\p_-M$. 

By Corollary  \ref{top}, we have that  each $\S_i$ is a disc or a $2$-sphere. The desired spheres, $\{S'_i\}_i$, will be obtained from surgeries on $\S_i$. More precisely, we define that 
\begin{itemize}
\item If $\Sigma_i$ is a sphere, we define $S'_i$ as $\Sigma_i$. 
\item If $\Sigma_i$ is a disc with $\p \Sigma_i\subset \p_0 M$,  we let $S'_i$ denote $\S_i\cup_{\p \S_i}D_i$, where $D_i\subset \p_0 M $ is an embedded disc with $\p \Sigma_i$. 
\end{itemize}The existence of $D_i$ is ensured by  the fact that $\p_0 M$ is a disjoint union of $2$-spheres.

\vspace{1mm}

We remark that since $D_i$ may be a subset of  $ D_{i'}$ for some $i'$, $S'_i$ may intersect $S'_{i'}$. If necessary, we deform $\{S'_i\}_i$ to be a disjoint collection $\{S_i\}_i$ of $2$-spheres  in $\i~ M$. They also  hold  that $M\setminus \amalg_i S_i$ has no component intersecting both $\p_+ M$ and $\p_- M$. 
\end{proof}
\subsection{Prime decomposition} We now use $\mu$-bubbles to complete the proof for Theorem \ref{C}. 
\begin{proof} Let $(M, g)$ be a complete orientable  open  $3$-manifold with scalar curvature $\kappa \geq 1$. Fix a point $x\in M$ and $\rho (\bullet) :=d(\bullet, x)$.  For our convenience, we may  assume that $\rho^{-1}(8\pi k)$ is  a smooth surface for each $k\in\mathbf{Z}_+$. They decompose $M$ into some compact manifolds $\{Y_k\}^\infty_{k=0}$ \[\begin{cases}
Y_k:=\rho^{-1}([8\pi k, 8\pi(k+1)])\\
 \p_+Y_k:=\rho^{-1}(8\pi k)\\ 
 \p_- Y_k:=\rho^{-1}(8\pi(k+1)).\end{cases}\]  Then, we have that for $k>0$, $\p Y_k=\p_-Y_k\amalg \p_+Y_k$ and $d(\p_- Y_k, \p_+ Y_k)>4\pi$. 
 
 \vspace{2mm}

 Lemma \ref{sep} allows to find  a finite collection $\{S_i\}_{i\in I_k}$ of disjoint $2$-spheres with the separation property:  
\begin{center} $Y_k\setminus \amalg_{i\in I_k}S_i$ has no component intersecting both $\p_+ Y_k$ and $\p_- Y_k$.  \quad  $\textbf{(S)}$ \end{center}
Then, we study the topological properties of all these spheres: 
 \begin{itemize}
\item the family, $\{S_i\}_{i\in  I}$, is locally finite, where $I:=\amalg_{k\in \mathbb{Z}_+}I_k$  (i.e. any compact set intersects (at most) finitely many elements in $\{S_i\}_{i\in I}$); 
\item the family $\{S_i\}_{i\in I}$ cuts $M$ into some connected $3$-manifolds, $\{X_l\}$;
\item the boundary of $X_l$ consists of $2$-spheres \end{itemize}

We have that each $X_l$ is contained in $\rho^{-1}([8\pi k, 8\pi (k+2)])$ for some $k\in \mathbb{Z}$. If not, one of components of $Y_k\cap X_l$ is a component of $Y_k\setminus\amalg_{i\in I_k}S_i $ intersecting both $\p_-Y_k$ and $\p_+ Y_k$ for some $k$ and $l$, which leads to a contradiction with the  separation property (see $\textbf{(S)}$).

\vspace{2mm}

Let $\hat{X}_l$ be a closed $3$-manifold obtained from $X_l$ by gluing some $3$-balls  along $\p X_l$.  The prime decomposition (see Theorem 1.5 of \cite{HA} on Page 5) gives a finite collection $\{S_i\}_{i\in I'_l}$ of disjoint $2$-spheres so that each component of $\hat{X}_l\setminus \amalg_{i\in I'_l}S_i$ is homeomorphic to a closed  prime $3$-manifold with finitely many punctures. 

\vspace{2mm}

We may assume that each $S_i$ is contained in $\i~X_l$ for each $i\in I'_l$. For our convenience,  we write $\{S_e\}_{e\in E}$, $\{M_v\}_{v\in V}$ for $\{S_i\}_{i\in I} \amalg_{l}\{S_i\}_{i\in I'_l}$ and the components of $M\setminus \amalg_{e\in E}S_e $. Each $M_v$ is homeomorphic to a closed prime $3$-manifold with finitely many punctures. 

There is a locally finite graph $G:=(V, E)$ constructed as follows: 
\begin{itemize}[leftmargin=15pt, topsep=1pt ]
\item Two distinct vertices, $v, v'\in V$ are connected by the edge $e\in E$ if both $\p M_v$ and $\p M_{v'}$ contain the $2$-sphere $S_e$;
\item The edge $e\in E$ connects the vertex $v$ with itself, if $\p M_v$ has two components coming from $S_e$. 
\end{itemize}
Thus, $M$ is an infinite connected sum of some closed prime $3$-manifolds. 
\end{proof}
We use free boundary functional  to generalize the result to the manifold with boundary. 

\begin{corollary}\label{mean-prime}Let $(M, g)$ be an orientable $3$-manifold with scalar curvature $\kappa\geq 1$.  If $\p M$ is  a disjoint union of some $2$-spheres and minimal for $g$, then  $M$ is an infinite connected sum of some closed prime $3$-manifolds and $3$-balls. 
\end{corollary}

By Theorem 3.7 of  \cite{Bar-Hanke} or Proposition 1.4 of Carlotto and Li \cite{C-L}, it can be generalized to mean convex boundary.

\section{Minimal surfaces and their existence}

In this section, we talk about minimal surface theory and the existence  of minimal $2$-spheres and discs. 

\subsection{Minimal surfaces and a metric inequality} The geometries of minimal surfaces are constrained by the scalar curvature (see \cite{SY2} and \cite{GL}).  We begin with a metric inequality:

\begin{proposition}\label{dist}\textnormal{(See Theorem 2 of \cite{SY2}  on Page 576 and Theorem 1 of \cite{Rosen} on Page 228)} Let $(M^3, g)$ be a 3-manifold and $(\S, \partial \S)\subset (M, \partial M)$ a stable minimal surface. If the scalar curvature $\kappa\geq 1$, then one has that, for any point $x\in \S$, 
\begin{equation*} 
d(x, \partial \S)\leq \frac{2\pi}{\sqrt{3}}. 
\end{equation*}
where $d$ is the distance function on $(M, g)$. 
\end{proposition}

\begin{corollary}\label{dist1} Let $(M, g)$ be a $3$-manifold with boundary and $(\S, \p\S)\subset (M, S_1\amalg S_2)$  a connected stable minimal surface with $\p \S\cap S_j\neq\emptyset$ for $j=1,2$, where $S_1\amalg S_2\subset \p M$. If   the scalar curvature $\kappa\geq 1$, then $d(S_1, S_2)\leq \frac{4\pi}{\sqrt{3}}$.  
\end{corollary}

\subsection{Embedded minimal disc and mean convex boundary} In the following, we will repeatedly use the geometric version of loop lemma. 
\begin{theorem}\label{disc}\textnormal{(See \cite{Meeks-Yau} and Theorem 6.28 of  \cite{CM1} on Page 224)} Let $(X^{3}, g)$ be a compact Riemannian 3-manifold whose boundary is mean convex and $\gamma$ a simple closed curve in $\partial X$ which is homotopically trivial  in $X$. Then, $\gamma$ bounds an area-minimizing disc and any such least area disc  is embedded.
\end{theorem}

The existence of minimal surfaces always requires  a geometric condition, mean convex boundary. We  introduce  a metric deformation to obtain mean convex boundary. 

\begin{lemma}\label{mean-convex-deform} Let $(X^3, g)$ be a compact $3$-manifold  with non-empty boundary. Then, for any $\epsilon>0$, there is a new metric $g_{\epsilon}$ so that 
\begin{itemize}
\item[(1)] $g_{\epsilon}$ is equal to $g$ on $X\setminus B(\p X, \epsilon)$;
\item[(2)] the mean curvature of $\p X$ is non-negative for $g_{\epsilon}$,  
\end{itemize} where $B(\p X, \epsilon)$ is the tubular neighborhood of $\p X$ with radius $\epsilon$.
\end{lemma}
\begin{proof}

Let $h(t)$ be a positive smooth function on $\mathbb{R}$ so that $h(t)=1$, for all $t\notin  [-\epsilon/2, \epsilon/2]$. Consider the function $f(x):=h(d(x, \partial X))$ and the metric $g_{\epsilon}:=f^{2}g$. In $(X, g_{\epsilon})$, the mean curvature $\hat{H}(x)$ of $\partial X$ is $$\hat{H}(x)=h^{-1}(0)(H(x)+2h'(0)h^{-1}(0))$$  
Choosing a function $h$ with $h(0)=2$ and $h'(0)>2\max_{x\in\partial X} |H(x)|+2$, we get the metric $g_{\epsilon}$ which is the required candidate in the assertion.\end{proof}

\subsection{Isotopy classes and Minimal surfaces}We  talk about how to use the isotopy class of $2$-spheres to find a minimal surface. 

\begin{definition}\label{isotopy} Let $\mathscr{C}_0$ be the collection of compact embedded surfaces such that each component is a $2$-sphere. Given $\Sigma\in \mathscr{C}_0$,  the  \emph{isotopy class} $\mathscr{I}^M(\S)$ of $\S$ in $M$ is the collection of all $\S'\in \mathscr{C}_0$ such that $\S'$ is  isotopic to $\S$ via a smooth isotopy $\phi_\bullet: [0,1]\times M\rightarrow M$, where $\phi_0=\textbf{Id}_M$ is the identity map  and each $\phi_t$ is a diffeomorphism of $M$. 

 When $M$ is non-compact, we also require that for an isotopy $\phi_\bullet$, there is a compact set $K\subset M$ such that $\phi_t |_{M\setminus K}=\textbf{Id}_{M\setminus K}$ is the identity map on $M\setminus K$ for each $t$.

\end{definition}

\begin{theorem}\label{mini-seq} \textnormal{(See Theorem 1 and Theorem 1' in \cite{MSY})}  Let $(M, g)$ be a compact orientable $3$-manifold with mean convex boundary and $\Sigma\subset M$ an embedded $2$-sphere. Consider an area-minimizing sequence $\{\Sigma_i\}^\infty_{i=1}$ in $\I^{M}(\S)$. If $\S$ is not homotopically trivial in $M$, then after passing to a subsequence, there are positive integers, $R$, $n_1$,$\cdots$, $n_R$ and pairwise disjoint stable minimal surfaces $\Sigma^{(1)}$, $\cdots$, $\Sigma^{(R)}$ such that 
\[\Sigma_k\rightarrow n_1\S^{(1)}+n_2\S^{(2)}+\cdots+n_R\S^{(R)}\]
in the sense of varifolds. Moreover, $g(\S^{(i)})=0$ for $1\leq i\leq R$. \end{theorem}

\begin{remark}\label{bdy} The limit surfaces enjoy the following properties (see Pages 624-625 of \cite{MSY}). 

(a)  Each $\Sigma^{(i)}$ is  $\mathbb{S}^2$ or 1-sided $\mathbb{R}P^2$. Moreover,  if $\Sigma^{(i)}$ is 1-sided $\mathbb{R}P^2$, its multiplicity $n_i$ should be even. 

(b)   If $\S^{(i)}$ is $\mathbb{S}^2$, then $B(\S^{(i)},\epsilon)$  is $[0, 1]\times \S^{(i)}$  for $\epsilon$ sufficiently small; if $\S^{(i)}$ is 1-sided $\mathbb{R}P^2$ then $ B(\S^{(i)},\epsilon)$ is $\mathbb{R}P^3\setminus B^3$.

(c) If $\Sigma^{(i)}$ is 1-sided $\mathbb{R}P^2$, the boundary $\tilde{\S}^{(i)}$ of the closure of $(M\setminus \S^{(i)}, g|_{M\setminus \S^{(i)}})$, coming from $\Sigma^{(i)}$, is a stable minimal $2$-sphere.  Thus, the boundaries of $(M\setminus \coprod\limits^R_{i=1}\Sigma^{(i)}, g|_{M\setminus \coprod\limits^R_{i=1}\S^{(i)}})$, coming from $\amalg^R_{i=1}\Sigma^{(i)}$, are a disjoint union of stable minimal $2$-spheres.

\end{remark}

\subsection{ $\gamma$-Reduction} In this part, we study the minimal surfaces (introduced by Theorem \ref{mini-seq}) and its relationship with $\I^M(\S)$. Their relationship is clarified by the so-called $\gamma$-reduction. See the general definition in \cite{MSY}. 

\begin{definition}\label{reduction} (See Pages 628-629 in \cite{MSY})Let $\S_1$, $ \S_2$ be two surfaces whose components are $2$-spheres and $\gamma$  be a positive constant.  We say $\S_2$ is a \emph{$\gamma$-reduction} of $\S_1$, if the following conditions are satisfied: 
\begin{itemize}[leftmargin=15pt]
\item[(1)]$\S_1\setminus \S_2$ has closure $A$ diffeomorphic to the closed annulus;
\item[(2)] $\S_2\setminus \S_1$ has closure consisting of two disc, $D_1$ and $ D_2$;
\item[(3)] $\p A=\p D_1\cup \p D_2$ and   $\text{Area}(A)+\text{Area} (D_1)+\text{Area}( D_2)\leq \gamma$;
\item [(4)]the sphere $A\cup D_1\cup D_2$ bounds a $3$-ball $Y$ with $\i~ Y\cap (\S_1\cup\S_2)=\emptyset$;
\item[(5)] Consider the component $\S^*_1$ of $\S_1$ containing $A$. The set $\S^*_1\setminus A$ is not connected and each component  has area $\geq \delta^2/2$, where  $\delta$  only depends on  $(M, g)$.
\end{itemize}\end{definition}

We note that that since $\S^*_1$ is a sphere, each component of $\S^{*}_1\setminus A$ is a disc.

\begin{figure}[H]
\centering{
\def\svgwidth{\columnwidth}
{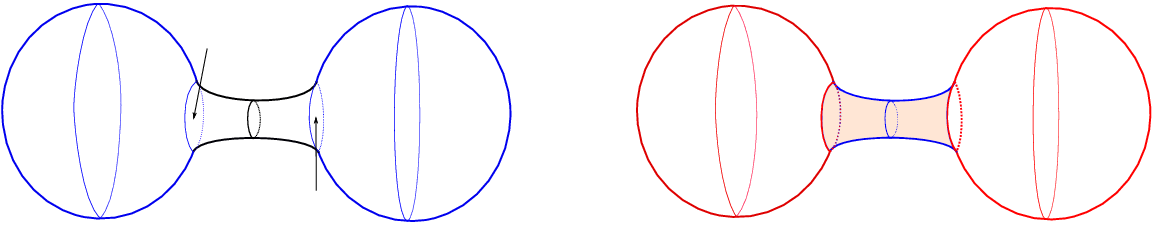}
\caption{A $\gamma$-reduction}
\label{reduce}
}
\end{figure}

\begin{remark}\label{red} (See Remark 3.27 of \cite{MSY} on Page 635) Let $\S_k$ and $\S^{(i)}$ be as in Theorem \ref{mini-seq}. For $k$ large enough and $1\leq i\leq R$, we define that 
\begin{equation*}\setlength{\abovedisplayskip}{1pt} \setlength{\belowdisplayskip}{1pt} S^{(i)}_k=\left\{  
\begin{array}{cc}
\amalg^{m_i}_{l=1} \p B(\Sigma^{(i)}, {l}/{k}), \quad\quad\quad\quad &\text{if}~n_i=2m_i~\text{is even}\quad\quad\\[3mm]
\Sigma^{(i)}\cup(\amalg^{m_i}_{l=1}\p B(\Sigma^{(i)}, {l}/{k})),\quad  &\text{if}~n_i=2m_i+1~\text{is odd}.
\end{array}
\right.\end{equation*}

There is a constant $\gamma$  depending on $(M, g)$, which satisfies the following conditions:   for $k$ large enough,  there is  a finite collection $\tilde{\Sigma}_k$ of disjoint $2$-spheres such that 

(1) it is obtained from $\S_k$ by a sequence of $\gamma$-reductions; 

(2) it is isotopic to $S_k:=\amalg^R_{i=0}S^{(i)}_k$ in $M$, where each sphere in $S^{(0)}_k$ is homotopically trivial in $M$. 
\end{remark}

\subsection{Irreducible manifolds and $\gamma$-reduction} In the following,  we let  $(M, g)$, $\S$, $\{\S_k\}$ and $\Sigma^{(i)}$ be assumed in Theorem \ref{mini-seq}. We now consider the case when $\Sigma$ (assumed as in Theorem \ref{mini-seq}) bounds a compact manifold $M_0\subset M$, homeomorphic to an irreducible $3$-manifold with a puncture.

\begin{definition}\label{bdy-fill} Let $X$ be a compact $3$-manifold whose boundary is a union of $2$-spheres. Its boundary filling  $\hat{X}$ is a closed 3-manifold obtained from $X$ by filling the boundary with some 3-balls. 

\end{definition}

\begin{lemma}\label{pun} Let $\S_1$ and $\S_2$ be two finite collections of disjoint $2$-spheres in a compact $3$-manifold $(M, g)$ (possibly with boundary), where $\S_2$ is a $\gamma$-reduction of $\S_1$. If $\S_1$ satisfies that 
\begin{itemize}[leftmargin=15pt]\item  $\S_1$ bounds a compact $3$-manifold $M_1\subset M$;
 \item the  boundary filling $\hat{M}_1$ of $M_1$ is  irreducible and $\pi_1(\hat{M}_1)\neq \{1\}$,
\end{itemize}
 then there is a compact component of $M\setminus \S_2$ whose boundary filling is homeomorphic to $\hat{M}_1$. 
\end{lemma}

\begin{proof} Let $A$, $Y$, $D_1$, $D_2$ and $\Sigma^{*}_1$ be as in Definition \ref{reduction}.  We assume that $\S^{i}_2$ is the component of $\S_2$ containing $D_i$ for $i=1,2$. Thus, $\S_1\setminus \S^*_1=\S_2\setminus \{\S^{i}_2\}_{i=1,2}$.

\vspace{2mm}

We note that  $\i Y\cap \S_1=\emptyset$ and $\partial Y\cap \Sigma_1=A$ (see (1) in Definition \ref{reduction}). Then, we have three distinct cases:  (1) $Y\cap M_1=\emptyset$; (2) $Y\cap M_1=A$ and $\i Y\cap M_2=\emptyset$;  (3) $Y\subset M_1$.
 
 \vspace{2mm}

\noindent \textbf{Case I}: If $Y\cap  M_1=\emptyset$,  we have that $\partial M_1 \subset \Sigma_2$. The reason is as follows: 

By that $A\subset \partial Y$ (see (3) in Definition \ref{reduction}), we find that $A\cap \partial M_1=\emptyset$, which implies  that $\partial M_1 \cap \S_1^*$ is also empty. Using $\S_1\setminus \S^*_1\subset \Sigma_2$ (see (1) in Definition \ref{reduction}), we get that $\partial M_1 \subset \Sigma_2$.

Thus, we obtain that $M_1$ is one component of $M\setminus \S_2$, a required candidate in the statement. 

\vspace{2mm}

\noindent \textbf{Case II} : If $Y\cap M_1=A$ and $\i Y\cap M_2=\emptyset$, we have that $\partial (M_1\cup Y) \subset \partial \Sigma_2$. The reason is as follows: 

We recall that  $\Sigma^*_1$ is the component of $\S_1$ containing $A$ (see (5) in Definition \ref{reduction}).  Then $\Sigma^*_1\subset \partial M_1$ (i.e. $\p M_1= \Sigma^*_1\cup  (\p M_1\setminus \S^*_1)$).  Thus, we find that
$$\p (M_1\cup_A Y)= ((\Sigma^*_1\setminus A)\cup D_1 \cup D_2)\cup(\p M_1\setminus \S^{*}_1)=\Sigma^1_2\cup\Sigma^2_2 \cup (\partial M_1\setminus \S^*_1)\subset \Sigma_2$$
The first equality comes from (4) and (5) in Definition \ref{reduction} while  the second equality  follows from (2) in Definition \ref{reduction}. 

\vspace{2mm}

Thus, we have that $M_1\cup Y$ is a component of $M\setminus \Sigma_2$ and define $M_2$ as $M_1\cup Y$. It remains to show that its boundary filling is $\hat{M}_1$. 

Consider the manifold $(M_1\cup Y)\cup_{\S^1_2} B_1 \cup_{\S^2_2} B_2$, where $B_i$ are two $3$-balls. Then, $B_i\cap Y$ is a disc $D_i$. We observe  that $Y\cup B_1\cup B_2$ is a $3$-ball, which bounds the sphere $\Sigma^*_1\subset M_1$. It implies  that $M_2=M_1\cup Y$ is a subset of $\hat{M}_1$. Therefore, we can conclude that the boundary filling of $M_2$ is $\hat{M}_1$. 

\vspace{2mm}

\noindent \textbf{Case III}:  If $Y\subset M_1$, we have $A\subset \p M_1$ . If not, we find that $A\cap \i M_1\neq \emptyset$, which implies that  $\Sigma_1^*\cap \i M_1\neq \emptyset$. It is in  contradiction with the fact $\i M_1\subset M\setminus \Sigma_1$.

We recall that (a) $\Sigma_2\setminus \Sigma_1$ has two components, $D_1$ and  $ D_2$;  (b) $\p Y=A\cup D_1 \cup D_2$ (see (2) and (4) in Definition \ref{reduction}).  Combining with the fact that $A\subset \p M_1$, we have that $\Sigma^1_2$ and $\Sigma^2_2$ is contained in $M_1$, which implies that   $M_1 \cap \Sigma_2=(\p M_1\setminus \Sigma^*_1)\amalg \{\Sigma^i_2\}_{i=1,2}\subset \hat{M}_1$ is a union of some $2$-spheres. 

Because of the irreducibility of $\hat{M}_1$, each sphere in $\Sigma_2\cap M_1$ bounds a $3$-ball in $\hat{M}_1$. By Van-Kampen's Theorem, there is a component $M_2$ of $\hat{M}_1\setminus (M_1\cap \Sigma_2)$ with $\pi_1 (M_2)\cong \pi_1 (\hat{M}_1)$. We observe that $\hat{M_1}$ can be written as  the connected sum of the boundary filling $\hat{M}_2$ of $M_2$ with some $3$-manifolds. 

Recall that any irreducible $3$-manifold is prime. We have that $\hat{M}_1$ is prime, which implies that $\hat{M}_2\cong \hat{M_1}$. Therefore, we can conclude that the boundary filling of $M_2$ is $\hat{M}_1$.  
\end{proof}

We now use Lemma \ref{pun} to study the spheres in Theorem \ref{mini-seq}.

\begin{proposition}\label{top-replace} Let $(M, g)$, $\S$, $\{\S_k\}$ and $\Sigma^{(i)}$ be as in Theorem \ref{mini-seq}. If $\S$ satisfies that 
\begin{itemize}[leftmargin=15pt]
\item the sphere $\S$ bounds a compact $3$-manifold $M_0\subset M$; 
\item the boundary filling $\hat{M}_0$ of $M_0$ is irreducible and it  differs from  $\mathbb{R}P^3$. 
\end{itemize}
then there is a component $M'$ of $M\setminus \amalg^R_{i=1}\Sigma^{(i)}$ whose boundary filling is homeomorphic to $\hat{M}_0$.  \end{proposition}

 Since $\S$ is not homotopically trivial in $M$, then ${M_0}$ is not a $3$-ball. Namely, $\hat{M}_0$ is not homeomorphic to $\mathbb{S}^3$.  Each $\S_k\in \I^M(\S)$ gives  a diffeomorphism $\phi_k$ of $M$  with $\phi_k(\S)=\S_k$. Thus, it also bounds $\phi_k(M_0)$, diffeomorphic to $\hat{M}_0$ with a puncture. 

\begin{proof}  Let $\tilde{\S}_k$, $m_i$ and $S_k$ be as in Remark \ref{red}. For $k$ large enough, each sphere in $S^{(0)}_k$ is homotopically trivial in $M$. To simplify our proof,   we may assume that $S^{(0)}_k$ is empty.

Since $\tilde{\S}_k$ is obtained by a sequence of $\gamma$-reductions from $\S_k\in \I^M(\S)$, we inductively use Lemma \ref{pun} to find a component $M_1$ of $M\setminus \tilde{\S}_k$  whose boundary filling is homeomorphic to $\hat{M}_0$. Since $\tilde{\S}_k$ is isotopic to $S_k$ for $k$ large enough, there is a diffeomorphism $\phi$ of $M$ with $\phi(\tilde{\S}_k)=S_k$. Therefore, there is a component $M'_1=\phi(M_1)$ of $M\setminus S_k$ whose boundary filling is homeomorphic to $\hat{M}_0$. 

Since each $\S^{(i)}$ is $\mathbb{S}^2$ or 1-sided $\mathbb{R}P^2$, then $\i B(\Sigma^{(i)}, {m_i}/{k})$ is homeomorphic to $\mathbb{S}^2\times (0,1)$ or $\mathbb{R}P^3$ with a puncture for $k$ large enough. In addition, each component of $\i B(\S^{(i)}, {m_i}/{k})\setminus \S^{(i)}$ is a cylinder. 

\vspace{2mm}

We have that $M'_1\cap \amalg_i B(\S^{(i)}, {m_i}/{k})=\emptyset$. (If not, $M'_1\subset \i B(\Sigma^{(i)}, {m_i}/{k})$ for some $i$. Then, from the last paragraph,  the boundary filling of $M'_1$ is  homeomorphic to  $\mathbb{S}^3$ or $\mathbb{R}P^3$, a contradiction.)

Choose $M'$ as the component of $M\setminus \amalg_{i=1}^R\S^{(i)}$ containing $M'_1$. The set $\i (M'\setminus M'_1)$ is  a subset of $\amalg_i~\i (B(\S^{(i)}, m_i/k)\setminus \S^{(i)})$ (each component is a cylinder). We have that $M'$ is homeomorphic to $M'_1$ (or $M_1$). That is to say, the boundary filling of $M'$ is homeomorphic to $\hat{M}_0$. 
\end{proof}

\begin{remark}\label{punctures} (1) We have that $\pi_1({M}')$ is isomorphic to $\pi_1(M_0)$. Since the fundamental group remains unchanged under the $\gamma$-reduction, then the isomorphism comes from an isotopy of diffeomorphism of $M$ . 

\vspace{2mm}

\noindent (2) We have that $\pi_1(M')$  is non-trivial. The reason is as follows: 

Since $\S$ is not homotopically trivial in $M$, then $M_0$ is not $3$-ball (i.e. $\hat{M}_0$ is not simply-connected). From (1), we have that $\pi_1(M')\cong \pi_1(\hat{M}_0)\neq \{1\}$. 

\vspace{2mm}

\noindent (3) Let $S\subset M'$ be an embedded $2$-sphere. Then there is a unique component of $M'\setminus S$ which is homeomorphic to a punctured $3$-sphere. 

\vspace{2mm}

The sphere $S$ can be considered as a 2-sphere in the irreducible $3$-manifold $\hat{M}'$, where $\hat{M}'$ is the boundary filling of $M'$.  It bounds a $3$-ball $B\subset \hat{M}'$. Thus, $B\cap M'$ is a punctured $3$-sphere. It is a component of $M'\setminus S$.  

It remains to show the uniqueness.  Let $M_1$ and $M_2$ be two components of $M'\setminus S$. If they are two punctured $3$-spheres, then $\pi_1(M_i)=\{1\}$ for $i=1, 2$.  Van-Kampen's theorem shows that $\pi_1(M')\cong\pi_1(M_1)\ast\pi_1(M_2)$ is trivial, a contradiction with the fact that $\pi_1(M')\neq \{1\}$. 
\end{remark}

\begin{corollary}\label{contractible} Let $M'$ and $M_0$ be assumed as in Lemma \ref{top-replace}. Then any closed curve $\gamma\subset M'\setminus M_0$ is contractible in $M'$. 
\end{corollary}

\begin{proof} Suppose that  a closed curve $\gamma\subset M'\setminus M_0$ is  non-contractible in $M'$. Since $\pi_1(M')$ is isomorphic to $\pi_1(M_0)$  and the isomorphism comes from the isotopy (see Remark \ref{punctures}), then there is a circle $\gamma'\subset M_0$ and an annulus $A\subset M$ so that \begin{itemize}
\item $\gamma'$ is not contractible in $M_0$; 
\item $\partial A=\gamma\amalg \gamma'$.  
\end{itemize}

We may assume that $\S:=\p M_0$ intersects $A$ transversally and $A\cap \S$ is a disjoint union of circles $\{c_l\}_{l}$.  The component $A_0$ of $A\cap M_0$ containing $\gamma'$ is a disc with finitely punctures with $\p A_0\subset \gamma' \amalg_l c_l$. Since $\S$ is a $2$-sphere, then each $c_l\subset \S$ is contractible in $\overline{M}_0$. So $\gamma'$ is contractible in $\overline{M}_0$. However, it is non-contractible in $M_0$, a contradiction.   
\end{proof}

\section{Isotopy classes and Mixed boundary} In this section, we introduce the mixed boundary condition  and show the existence of $3$-submanifold with mixed boundary in a complete $3$-manifold with scalar curvature $\kappa\geq 1$ (see Theorem \ref{mixed-existence}). 

\begin{definition}\label{mixed}An orientable  $3$-manifold $(X, g)$ has \emph{mixed boundary} if it satisfies that 
\begin{itemize} \item[(a)] the boundary $\p X$ consists of  two mutually disjoint families $\{S_j\}_{j\in J'}$ and $\{S_j\}_{j\in J''}$, of 2-spheres;
 \item[(b)] each $S_j$ is minimal for $g$, where $j\in J'$;
 \item[(c)] there are disjoint punctured $3$-spheres $\{\O_l\}_l$ whose boundaries contain  $\amalg_{j\in J''} S_j$;
 \item[(d)] the metric $g$ satisfies  that for each $l$,
 $$d(\hat{S}_l, \p \O_l\cap \amalg_{j\in J''}S_j)>4\pi,$$
where $\hat{S}_l:=\p\O_l\cap \i~X$ is a disjoint union of some $2$-spheres.
\end{itemize}
\end{definition}

For example, if the boundary of  $3$-manifold is a union of minimal  $2$-spheres, it has mixed boundary. 

\begin{remark}\label{cover} If $(X, g)$ has mixed boundary, then the covering space $(\tilde{X}, \tilde{g})$ also has mixed boundary. 
\end{remark}

\begin{theorem}\label{mixed-existence}Let $(M, g)$ be a complete  open and  connected  orientable $3$-manifold with scalar curvature $\kappa\geq 1$ and $M_0\subset M$ be   a compact subset which satisfies the following conditions  \\
(a) the boundary $\S:=\p M_0$ is a $2$-sphere;\\
(b) the boundary filling $\hat{M}_0$ of $M_0$ is irreducible; \\
(c) $\hat{M}_0$ is neither $\mathbb{S}^3$ nor $\mathbb{R}P^3$. \\
Then, there is an open  $3$-manifold $M'\subset M$ so that (1)
 $M'$ is homeomorphic to $\hat{M}_0$ with finitely many punctures;
 (2) the closure of $(M', g|_{M'})$ is a compact $3$-manifold with mixed boundary.

\end{theorem}

\begin{remark}\label{non-trivial-sphere} Let $M$ and $\Sigma$ be assumed in Theorem \ref{mixed-existence}. Then, we can conclude that $\Sigma$ is not homotopically trivial in $M$. The reason is as follows: 

If not, Proposition \ref{ball} shows $\S$ bounds a $3$-ball $B\subset M$. Then, by the assumption that $\hat{M}_0$ is not a $3$-sphere,   we find that $M_0$ is not a $3$-ball and $M_0\cap B=\p \S$, which implies that $M_0\cup_\S B\subset M$ is a compact $3$-manifold without boundary. It is  in contradiction with the assumption  that $M$ is connected and open.

\end{remark}

We first construct the required manifold $M'$ as given in 
Lemma \ref{existence} and then show the manifold $(M', g|_{M'})$ has mixed boundary in Lemma \ref{bdy-proof}. 
\subsection{Find the required manifold $M'$}For a compact set $K\subset (M, g)$, we introduce two subsets about $K$: 
 \[B^M(K, r_1):=\{x~|~d(x, K)\leq r_1\} \text{ and } A^M(K, r_1, r_2):=\overline{B^M(K,r_2)\setminus B^M(K, r_1)}, \] where $r_2>r_1$.

\begin{lemma}\label{existence} Let $(M, g)$, $\Sigma$ and $M_0$ be assumed as in Theorem \ref{mixed-existence}. Then there is an open $3$-manifold $M'\subset M$ satisfying that 
\begin{itemize}
\item[(1)] $M'$ is homeomorphic to $\hat{M}_0$ with finitely many punctures; 
\item[(2)] the boundary of $(M', g|_{M'})$ consists of two disjoint families $\{S_j\}_{j\in J'}$ and $\{S_j\}_{j\in J''}$ of $2$-spheres;
\item[(3)] For $j\in J'$, $S_j$ is stable minimal for $g$, while for $j\in J''$, $S_j$ satisfies 
\[d(S_j, M_0\cap M')\geq 16\pi, \]
where $d$ is the distance function induced from the closure of $({M}', g|_{{M}'})$. \end{itemize}
\end{lemma}

\begin{proof}

 We may assume that the boundary of  the compact set $B^{M}(M_0, 18\pi+\epsilon)$ is smooth for some $\epsilon$. Lemma \ref{mean-convex-deform} allows to find a metric $g'$   on $B^{M}(M_0, 18\pi+\epsilon)$ so that 
\begin{itemize}
\item[(a)] $g'$ is equal to $g$ on $B^M(M_0, 18\pi)$;
\item[(b)] the mean curvature of $\p B^M(M_0, 18\pi+\epsilon)$ is positive for $g'$.
\end{itemize} Then, $(B^{M}(M_0, 18\pi+\epsilon), g')$ has mean convex boundary.  However,  the scalar curvature of $g'$ may be negative on $A^{M}(M_0, 18\pi, 18\pi+\epsilon)$. 

\vspace{2mm}

We first observe that  $\Sigma$ is not homotopically trivial in $M$ (see Remark \ref{non-trivial-sphere}). Then we  consider an area-minimizing sequence $\{\Sigma_k\}$ of the isotopy class of $\S$  in the manifold $(B^M(M_0, 18\pi+\epsilon), g')$. 
The existence of the limit of the sequence is ensured by the condition that $(B^M(M_0, 18\pi+\epsilon), g')$ has mean convex boundary.

By Theorem \ref{mini-seq}, after passing to a subsequence,  there are positive integers, $R$, $n_1$,$\cdots$, $n_R$ and pairwise disjoint stable minimal surfaces $\Sigma^{(1)}$, $\cdots$, $\Sigma^{(R)}$ such that 
\[\setlength{\abovedisplayskip}{0pt} \setlength{\belowdisplayskip}{0pt}\Sigma_k\rightarrow n_1\S^{(1)}+n_2\S^{(2)}+\cdots+n_R\S^{(R)}\]
in the sense of varifolds.  The surfaces $\{\S^{(i)}\}_{i=1}^R$ are stable minimal for $g'$. Based on their geometries, they could be classified into two families: 
\[\begin{split}&I':=\{\Sigma^{(i)}~|~ \Sigma^{(i)}\subset B^M(M_0, 18 \pi)\} \\ &I'':=\{\Sigma^{(i)}~|~ \Sigma^{(i)}\setminus B^M(M_0, 18\pi)\neq \emptyset. \}\end{split}\]
The surface $\Sigma^{(i)}\in I'$ is stable minimal for $g'$, because $g=g'$ on $B^M(M_0, 18\pi)$. For $\S^{(i)}\in I''$, it intersects $A(M_0, 18\pi, 18\pi+\epsilon)$ and $\Sigma^{(i)}\cap B^M(M_0, 18\pi)$ is stable minimal for $g$. Because  the scalar curvature $\kappa_g\geq 1$,  Proposition \ref{dist} shows that for any $x\in \Sigma^{(i)}\cap B^M(M_0, 18\pi)$,  
$$d_0(x, \p B^M(M_0, 18\pi))\leq 2\pi/\sqrt{3},  $$ where $d_0$ is distance function on $(M, g)$.  Thus, we have that for any $\S^{(i)}\in I''$, $d_0(\Sigma^{(i)}, M_0)\geq 18\pi- 2\pi/\sqrt{3}>16\pi$.

\vspace{2mm}

From our assumption that  $\hat{M}_0$ is irreducible and differs from $\mathbb{R}P^3$,   we use Proposition \ref{top-replace} to  find an open $3$-manifold $M'\subset M$  that is homeomorphic to $\hat{M}_0$ with finitely many punctures. 

We note that 
 that $M'$ is one of components of $B^{M}(M_0, 18\pi+\epsilon)\setminus \amalg_{i=1}^R \S^{(i)}$ and its boundary consists of some $2$-spheres $\{S_j\}_{j\in J}$ coming from $\{\S^{(i)}\}_{i=1}^R$ (see Remark \ref{bdy}).  
 
 Using the classification for $\{\S^{(i)}\}$, the spheres $\{S_j\}_{j\in J}$ split into two groups. 
 \[\begin{split}& J':=\{j\in J ~|~ S_j \text{ comes from some sphere in } I'\}\\& J'':=\{j\in J ~|~ S_j \text{ comes from some sphere in } I''\}\end{split}\]
 We recall that any surface $\Sigma^{(i)}\in I'$ is stable minimal for $g$ while $\Sigma^{(i)}\in I''$ satisfies $d_0(\Sigma^{(i)}, M_0)>16\pi$. Then, we find that 
 \begin{itemize}
 \item[(a)] For $j\in J'$, $S_j$ is stable minimal for $g$ 
 \item[(b)] For $j\in J''$, we have that $d(M_0\cap M', S_j)>16\pi$. 
 \end{itemize}\end{proof}
 
%
%
%
%
%
%

\subsection{Mixed boundary and punctured $3$-spheres} In the following, we show that $(M', g|_{M'})$ constructed in Lemma \ref{existence} has mixed boundary. 

\begin{lemma}\label{bdy-proof} Let $(M, g)$ , $\S$, $M_0$ be assumed as in Theorem \ref{mixed-existence} and $M'$ constructed as in Lemma \ref{existence}. Then $(M', g|_{M'})$ has mixed boundary.  
\end{lemma}

\begin{proof} The open $3$-manifold $M'\subset M$ constructed in Lemma \ref{existence} satisfies the following properties:  
\begin{itemize}
\item[(1)] $M'$ is homeomorphic to $\hat{M}_0$ with finitely many punctures; 
\item[(2)] the boundary of $(M', g|_{M'})$ consists of two mutually different families $\{S_j\}_{j\in J'}$ and $\{S_j\}_{j\in J''}$
\item[(3)] For $j\in J'$, $S_j$ is stable minimal for $g$;
\item[(4)] For $j\in J''$, $d(S_j, M_0\cap M')\geq 16\pi$. 
\end{itemize} We note that for $j\in J'$, the diameter of $S_j$ is bounded up by $\frac{2\pi}{\sqrt{3}}$, because of $\kappa_g\geq 1$ and Proposition \ref{dist}. 

\vspace{2mm}

In the following, we first find a collection $\{\hat{S}_l\}$ of $2$-spheres  satisfying that no component of $M'\setminus \amalg_l \hat{S}_l$ meets both $M_0\cap M'$ and $\amalg_{j\in J''} S_j$. Then we use these $2$-spheres to find punctured $3$-spheres and to verify the mixed boundary condition. 

\vspace{1mm}

\noindent\textbf{Step 1}: Find the required $2$-spheres $\{\hat{S}_l\}\subset M'$. 

\vspace{1mm}

Before constructing $\{\hat{S}_l\}$, we should find a region $Y$ which satisfies the conditions in Lemma \ref{sep}. Then, we use Lemma \ref{sep}  to find $\hat{S}_l$. 

\vspace{1mm}

We will construct a region $Y\subset M$ and a subset $J_1\subset J'$ with the following properties: 
\begin{itemize}
\item[(a)] $A^{M'}(M_0\cap M', 4\pi, 8\pi )\subset Y\subset A^{M'}(M_0\cap M', 2\pi, 10\pi)$; 
\item[(b)] $\p Y=\p_+ Y\amalg \p_- Y\amalg\{S_j\}_{j\in J_1}$, where $J_1\subset J'$
\item[(c)] $d(\p_+ Y, \p_- Y)\geq 4\pi$.
\end{itemize}

\begin{figure}[H]
\centering{
\def\svgwidth{\columnwidth}
{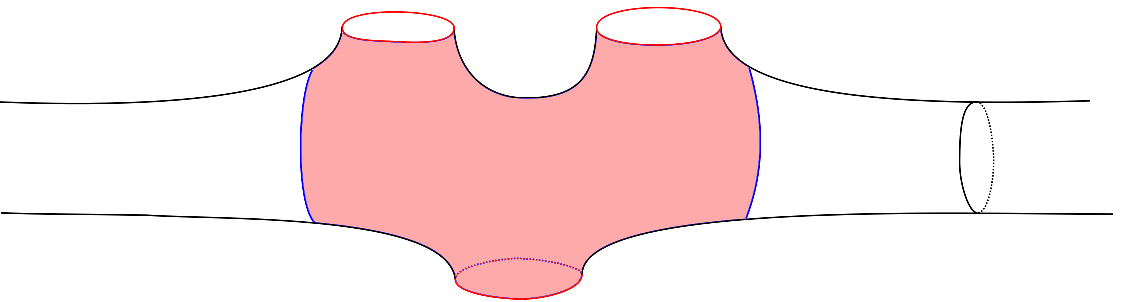}
\caption{}
\label{region}
}
\end{figure}

Let me explain the construction of $Y$ and $J_1$.  We define that 
\[J_1=\{j\in J'~|~ S_j\cap A^{M'}(M'\cap M_0, 4\pi, 8\pi)\neq \emptyset\}\]
We recall that for $j\in J'$, $S_j$ has the diameter $<2\pi$.  For $j\in J_1$, $S_j$ is a subset $\i A^{M'}(M'\cap M_0, 2\pi, 10\pi)$. 

For $j\in J_1$, we consider the tubular neighborhood $B^{M'}(S_j, \epsilon)$.  For $\epsilon_j$ small enough, we may assume that  $B^{M'}(S_j, \epsilon_j) \subset A^{M'}(M'\cap M_0, 2\pi, 10\pi)$ and  $\{B^{M'}(S_j, \epsilon_j)\}_{j\in J_1}$ are disjoint. We now define 
\[Y:=A^{M'}(M'\cap M_0, 4\pi, 8\pi)\cup_{j\in J_1} B^{M'}(S_j, \epsilon)\]
It is a subset of $A^{M'}(M'\cap M_0, 2\pi, 10\pi)$.

Let $\p_+ Y= \p Y \cap A^{M'}(M'\cap M_0, 2\pi,  4\pi)\cap \i M'$ and $\p_- Y=\p Y \cap A^{M'}(M'\cap M_0, 8\pi,  10\pi)\cap \i M'$. Then, we find that $\p Y=\p_+ Y \amalg \p_-Y\amalg_{j\in J_1} S_j$ and $d(\p_+ Y, \p_- Y)\geq 8\pi- 4\pi=4\pi$.

We apply Lemma \ref{sep} to the manifold $(Y, g|_{Y})$ and find a finite collection $\{\hat{S}_l\}$ of disjoint $2$-spheres in $\i Y$ so that
\begin{itemize}
\item no component of $Y\setminus \amalg \hat{S}_l$ intersects both $\p_+ Y$ and $\p_- Y$; 
\item $\hat{S}_l\subset Y\subset\i ~ B^{M'}(M_0\cap M',10\pi)$ for each $l$. 
\end{itemize}
Further, we have that no component of $M'\setminus \amalg_{l}\hat{S}_l$ meets both $M_0\cap M'$ and $\amalg_{j\in J''} S_j$. 
Otherwise, there is a component of $Y\setminus \amalg_{l} \hat{S}_l$ meeting both $\p_+ Y$ and $\p_- Y$, a contradiction with the fact that no component of $Y\setminus \amalg \hat{S}_l$ intersects both $\p_+ Y$ and $\p_- Y$. 

\vspace{2mm}

\noindent\textbf{Step 2}: Find disjoint punctured $3$-spheres, $\{\O_l\}_l$,  in Definition \ref{mixed}.  

\vspace{1mm}

For each $l$, there is a unique component $\O_l$ of $M'\setminus \hat{S}_l$ which is a punctured $3$-sphere (see Remark \ref{punctures}).  Since $\hat{S}_l\subset\i ~ B^{M'} (M_0\cap M', 10\pi)$ and $d(S_j, M_0\cap M')\geq 16\pi$ for $j\in J''$, we have that  $d(\hat{S}_l, \p \O_l\cap \amalg_{j\in J''}S_j)\geq 16\pi-10\pi>4\pi$. 

\vspace{2mm}

\noindent\textbf{Claim}: For each $l$ and $l'$, one of the following holds: (1) $\O_l\cap \O_{l'}=\emptyset$; (2) $\O_l\subset \O_{l'}$; (3) $\O_{l'}\subset \O_l$. 

\vspace{1mm}

Suppose that $\O_{l}\cap\O_{l'} $ and $\O_{l}\setminus \O_{l'}$ are both nonempty  (i.e. there are two points, $p_1\in \O_l\setminus \O_{l'}$ and $p_2\in \O_{l'}\cap \O_{l}$). It is sufficient  to show that $\O_{l'}\subset \O_l$. 

\vspace{1mm}

First, $\hat{S}_{l'}$ is a subset of $\O_l$. The reason is as follows: there is a curve $\gamma\subset \O_l$ connecting $p_1$ and $p_2$. The curve $\gamma$ intersects $\hat{S}_{l'}$ at some point. Namely, $\hat{S}_{l'}\cap \O_l\neq \emptyset$. 
Since $\hat{S}_l\cap \hat{S}_{l'}=\emptyset$, then $\hat{S}_{l'}$ is contained in one of components of $M'\setminus \hat{S}_l$. Thus , we have that $\hat{S}_{l'}$ is a subset of $\O_l$. 

\vspace{1mm}

We recall that  $\hat{S}_{l'}$ cuts $\O_l$ into two punctured $3$-spheres, since $\O_l$ is a punctured $3$-sphere. Then, take the component $B$ of $\O_{l}\setminus \hat{S}_{l'}$ with $\p B\cap \hat{S}_{l}=\emptyset$. It  is a component of $M'\setminus \hat{S}_{l'}$ and  a punctured $3$-sphere . From the uniqueness of $\O_{l'}$ (see Remark \ref{punctures}), we can conclude that $B$ is equal to $\O_{l'}$. Namely, $\O_{l'}\subset \O_l$. 

This completes the proof of the claim. 

\vspace{2mm}

Consider the partially ordered set $(\{\O_l\}_{l}, \subset)$. It is a finite set. In the following, we abuse the notions and write $\{\O_l\}_{l}$ for the maximal elements in $(\{\O_l\}_l, \subset)$. The sets $\{\O_l\}$ are pairwise disjoint.  

\vspace{1mm}

We notice  that $M'':=M'\setminus \amalg_l \O_l$ is a component of $M'\setminus\amalg_l \hat{S}_l$. Since each $\O_l$ is simply-connected,  we use Van-Kampen's theorem to have that $\pi_1(M'')\cong \pi_1(M')\neq \{1\}$. The isomorphism comes from the inclusion.  

\vspace{2mm}

\noindent\textbf{Step 3}:  Show that $\amalg_{j\in J''}S_j \subset \amalg_{l}\p \O_l$. 

\vspace{1mm}

We have that $M_0\cap M'\subset \amalg \O_l$ or $\amalg_{j\in J''} S_j\subset \amalg \O_l$.  If not, $M''\subset M\setminus \amalg_{l}\hat{S}_l$ meets $\amalg_{j\in J''}S_j$ and $M_0\cap M $, which is in contradiction with the fact that no component of $M'\setminus \amalg_l \hat{S}_l$ meeting both $M_0\cap M'$ and $\amalg_{j\in J''} S_j$. 

\vspace{1mm}

If $M_0\cap M'\subset \amalg_{l} \O_l$, we have that $M''$ is a subset of $M'\setminus M_0$.  There is a non-contractible circle $c\subset M''$ in $M'$ (since $\pi_1(M'')\neq \{1\}$).  However, Corollary \ref{contractible} shows that $c\subset M''\subset  M'\setminus M_0$ is contractible in $M'$, which leads to a contradiction. 

We can conclude that $\amalg_{j\in J''}S_j\subset \amalg_l \p \O_l$. Namely, $(M', g|_{M'})$ has mixed boundary.
\end{proof}

\begin{remark}\label{mean-mixed} Let $(M, g)$ be  an orientable (non-compact) $3$-manifold with scalar curvature $\kappa\geq 1$ whose boundary is a union of some $2$-spheres and minimal for $g$. Theorem \ref{mixed-existence} is also true for such $3$-manifolds. 
\end{remark}

\section{Classification of $3$-manifold with mixed boundary}

In this section, we  classify $3$-manifolds with mixed boundary and with uniformly positive scalar curvature. 

Let be $(M^3, g)$ a $3$-manifold $(M^3, g)$ with mixed boundary and  $\{S_j\}_{j\in J'}$,  $\{S_j\}_{j\in J''}$ and $\O_l$ be defined in  Definition \ref{mixed}. In what follows, we define  $J_l:=\{~j\in J~|~S_j\subset \p \O_l\}$, $J'_l:=J'\cap J_l$ and $J''_l:=J''\cap J_l$. Then, the term (d) in Definition \ref{mixed} can be rewritten as follows: 
\[d(\hat{S}_l,  \amalg_{j\in J''_l}S_j)>4\pi\]

\begin{figure}[H]
\centering{
\def\svgwidth{\columnwidth}
{\scalebox{0.6}{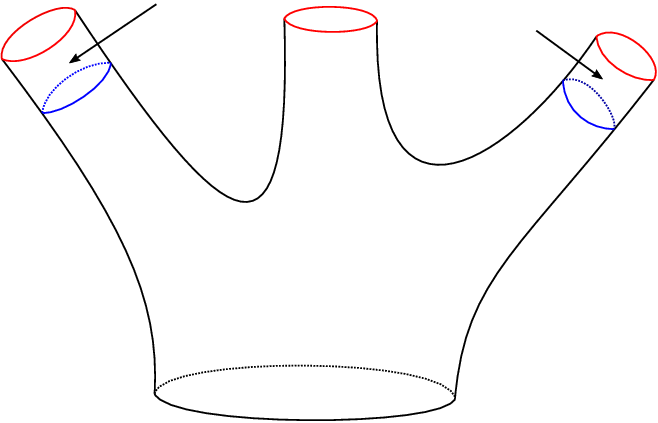}}
\caption{}
\label{metric}
}
\end{figure}

\begin{remark}\label{pi-2}Let $\hat{M}=M\cup_{j\in J} B_j$ be the boundary filling  of $M$, where $B_j$ is a $3$-ball. We have the following properties:

 (1)  If $\pi_2(M)$ is generated by the spheres in $\p M$, then $\pi_2(\hat{M})$ is trivial.

(2) Each $\hat{S}_l$ is also  a subset in $\hat{M}$ and  bounds a $3$-ball $\hat{\O}_l\subset \hat{M}$, where   $\hat{\O}_l:=\O_{l}\bigcup_{j\in J_l}(\cup_{S_j} B_j)$.  

(3) The $3$-balls, $\{\hat{\O}_l\}_l$ and $\{B_j\}_{j\in J'\setminus \amalg_l J_l}$ are disjoint in $\hat{M}$. 

\end{remark}

\begin{lemma}\label{new-metric} Let $(M^3, g)$ be a manifold with mixed boundary and $\{S_j\}_{j\in J'\amalg J''}$, $\{\O_l\}_l$ and $\hat{S}_l$ be assumed as above. Then, there is a metric $g'$ on $M$ with the following properties: 
\begin{itemize}
\item[(1)] $(M, g')$ has mean convex boundary; 
\item[(2)] For each $j\in J''$, there is a neighborhood $B(S_j , \epsilon_j)$ of $S_j$ in $M$ so that $\{B(S_j, \epsilon_j)\}_{j\in J''}$ and $\{S_j\}_{j\in J'}$ are disjoint; 
\item[(3)] $d_{g}(S'_j, \hat{S}_l)>4\pi$ for $j\in J''_l$; 
\item[(4)] $g'$ is equal to $g$ on $M\setminus \amalg_{j\in J''} B(S_j, \epsilon)$, 
\end{itemize}where $S'_j:=\p B(S_j, \epsilon_j)\cap \i M$ (see Figure \ref{metric}). 
\end{lemma}

\begin{proof} For each $j\in J''$, choose $\epsilon_j>0$ so that $\{B(S_j, \epsilon_j)\}_{j\in J''}$ and $\{S_j\}_{j\in J'}$ are disjoint. For $\epsilon_j$ sufficiently small, we may assume that for $j\in J''_l$, $d(\hat{S}_l, S'_j)>4\pi$  and $B(S_j, \epsilon_j) \subset \O_l$ for $j\in J''_l$, where $S'_j=\i M\cap \p B(S_j, \epsilon_j)$ (see Figure \ref{metric}). 

We apply Lemma \ref{mean-convex-deform} to each $B(S_j, \epsilon_j)$ and obtain a metric $g'$ on $M$ which satisfies  that 

(I) for $j\in J''$, $S_j$ is mean convex, with respect to $g'$; 

(II)  $g'=g$ on $M'\setminus \amalg_{j\in J''} B(S_j, \epsilon_j)$.

Combining (II) with the term (b) in Definition \ref{mixed}, we can conclude that  $(M, g')$ has mean convex boundary. 
\end{proof}

 We remark that if the scalar curvature of $(M, g)$ is uniformly positive, the scalar curvature of $g'$ may be negative on $B(S_j, \epsilon_j)$ for some $j$. 
 
 \vspace{2mm}

 In the following, we  study the fundamental group of a $3$-manifold with mixed boundary and with uniformly positive scalar curvature

\begin{lemma}\label{pre} Let $(M, g)$ be a $3$-manifold with mixed boundary satisfying that $\pi_2(M)$ is generated by the spheres in $\p M$. If $\pi_1(M)$ has a subgroup which is isomorphic to $\mathbb{Z}$,  then the scalar curvature $\kappa$ of $(M, g)$ can not be bounded below by one.
\end{lemma}

\begin{proof}
We may assume that $\pi_1(M)$ is isomorphic to $\mathbb{Z}$. If not, we choose a cover of $M$ whose fundamental group is isomorphic to $\mathbb{Z}$ and  replace $M$ by it. (By Remark \ref{cover}, its cover is also a manifold with mixed boundary.)

\vspace{1mm}

In the following, we assume that $B_j$, $\hat{M}$, $\hat{\O}_l$, $J_l$, $S'_j$ and $g'$ are defined as above. Suppose  the contrary that  the scalar curvature $\kappa$ of $g$ is bounded below by one (i.e. $\kappa\geq 1$). 

\vspace{1mm}

By Remark \ref{pi-2}, we have $\pi_2(\hat{M})=\{1\}$( i.e.$\hat{M}$ is irreducible). Since $\pi_1(\hat{M})$ is infinite, the universal cover of ${\hat{M}}$ is contractible (see Corollary 3.9 of  \cite{HA} on Page 51), which implies  $\hat{M}$ is a $K(\mathbb{Z}, 1)$-space. Namely, $\hat{M}$ is homotopically equivalent to $\s^1$. 

\vspace{2mm}

Let $\gamma\subset M$ be  a circle generating $\pi_1(\hat{M})$. We may assume that $\gamma$ is contained in $M\setminus\amalg_l \O_l$ (since each $\O_l$ is a punctured $3$-sphere). By Corollary \ref{required-curve}, there is a closed curve $\sigma\subset \hat{M}\setminus B(\gamma, 5\pi)$ so that it is contractible in $\hat{M}$ but non-contractible in $\hat{M}\setminus \gamma$. Moreover, 
$\sigma$ is a subset of $\hat{M}\setminus B(\gamma, 5\pi)$, which implies   $d(\sigma, \gamma)\geq 5\pi >4\pi$.

Since there are disjoint $3$-balls, $\{\hat{\O}_l\}$ and $\{B_{j}\}_{j\in J'\setminus \amalg_{l} J_l}$, we may assume that $\sigma$ is a subset of $M\setminus \amalg_l \O_l= \hat{M}\setminus (\amalg_l \hat{\O}_l \cup \amalg_{j\in J'\setminus\amalg_{l} J_l} B_j)$.  Combining with  Van-Kampen's theorem, we find $\pi_1(M)\cong \pi_1 (\hat{M})$ and $\pi_1(M\setminus \gamma)\cong \pi_1(\hat{M}\setminus \gamma)$, which implies the following properties: 
\begin{itemize} 
\item[(1)]$\gamma$ is homotopically trivial in $M$ ; 
\item[(2)] it is non-contractible in $M\setminus \gamma$. 
\end{itemize}

We choose the metric $g'$ constructed from Lemma \ref{new-metric}. The manifold $(M, g')$ has mean convex boundary. Then, we use the solution of Plateau problem to find an area-minimizing disc $D\subset M$ with boundary $\sigma$.  It is stable minimal for $g'$. 
\vspace{1mm} 

\noindent\textbf{Claim:} $D$ is also stable minimal for $g$. 

\vspace{1mm}

We recall that $g'=g$ on $M\setminus \amalg_{j\in J''} B(S_j, \epsilon_j)$. 
It is sufficient to show that $D$ is a subset of $M\setminus \amalg_{j\in J''} B(S_j, \epsilon)$.  Suppose the contrary that $D$ intersects some $B(S_{j_0}, \epsilon_{j_0})$ for some $j_0\in J''_l\subset J''$. Say, $D\cap {S}'_{j_0}\neq \emptyset$.

 Since $\sigma=\p D$ is away from $\O_l$,  there is a component $D_s$ of $D\cap (\O_l\setminus \amalg_{j\in J''_l} B(S_j, \epsilon_j))$  connecting $\hat{S}_l$ and $\amalg_{j\in J''_l} S'_{j}$.

Since the scalar curvature $\kappa$ of $(M, g)$ $\geq 1$, Corollary \ref{dist1} implies that $d(\amalg_{j\in J''_l}S'_{j}, \hat{S}_l)\leq 4\pi$. However, $d(\amalg_{j\in J''_l}S'_{j}, \hat{S}_l)>4\pi$ (see Lemma \ref{new-metric}), which leads to a contradiction. It completes  the proof of the claim.

\vspace{2mm}

We now use the claim to finish the proof. Since $\sigma$ is non-contractible in $M\setminus \gamma$, then $D\cap \gamma$ is non-empty. Consider the set $B(\gamma, 4\pi):=\{x\in M~|~ d(\gamma, x)\leq 4\pi\}$ and a component $D'$ of $D\cap B(\gamma, 4\pi)$ intersecting $\gamma$. The boundary $\p D'$ is a subset $\p B(\gamma, 4\pi)$.

 Because $\kappa\geq 1$, Theorem \ref{dist} gives that\[D'~\text{is contained in the} ~\frac{2\pi}{\sqrt{3}}\text{-neighborhood of }\p D' ~\text{in} ~D'. \]Namely, $d(\p D', \gamma)<2\pi$. 
However, $d(\gamma, \p D')\geq 4\pi$ (since $\p D'\subset \p B(\gamma, 4\pi)$), a contradiction. We complete the proof. \end{proof}

In the following, we use Lemma \ref{pre} to finish the proof for Theorem \ref{classify}
\vspace{1mm}

\noindent {\bf{Theorem \ref{classify}}} \emph{ Let $(M, g)$ be a compact $3$-manifold with scalar curvature $\kappa\geq 1$ and with mixed boundary. If $\pi_2(M)$ is generated by the spheres in $\p M$, then $\text{Int}~M$ is homeomorphic to a spherical $3$-manifold with finitely many punctures. }

\vspace{1mm}
\begin{rem} If $M$ is non-compact, it is homeomorphic to $\mathbb{R}^3$ with  (at most) countably many disjoint open $3$-balls removed. 
\end{rem}

\begin{proof}Let  $\hat{M}$ be obtained from $M$ by filling the boundary with some $3$-balls. Then  it is a closed  manifold and $\pi_2(\hat{M})=\{1\}$.    We  have that  $\pi_1(\hat{M})$ is finite. The reason is as follows:

 If not,  the universal cover of $\hat{M}$ is contractible (see Corollary 3.9 of \cite{HA} on Page 51).  Namely, the compact manifold $\hat{M}$ is a $K(\pi, 1)$-space and $\pi_1(\hat{M})$ is non-trivial. From Lemma 4.1 of \cite{luck} on Page 5,  $\pi_1(\hat{M})$ is torsion-free, which implies that any non-trivial element in $\pi_1(M)$ has infinite order. The subgroup generated by the element is isomorphic to $\mathbb{Z}$.  However, the scalar curvature $\kappa \geq 1$, which  contradicts with Lemma \ref{pre}.  
 
 \vspace{2mm}

 Since $\pi_1(\hat{M})$ is finite, the universal cover of $\hat{M}$ is compact. 
The  Poincar\'e conjecture (see \cite{BBBMP} ,\cite{MT} and \cite{cao-zhu}) shows that   the universal cover is homeomorphic to $\s^3$. Namely, $\hat{M}$ is a spherical $3$-manifold. Thus, $\i M$ is a spherical $3$-manifold with finitely many punctures. \end{proof}

Theorem \ref{C} shows that a complete orientable $3$-manifold with uniformly positive scalar curvature has a prime decomposition. We now use Theorem \ref{classify} to complete the proof of Theorem \ref{A}. It is sufficient to show that each prime factor in the decomposition is a spherical $3$-manifold or $\mathbb{S}^1\times \mathbb{S}^2$. 
\begin{proof} (The proof of Theorem \ref{A}) Let $(M, g)$ be a complete connected  orientable $3$-manifold with scalar curvature $\kappa\geq 1$. From Theorem \ref{C}, $M$ is homeomorphic to a (possibly) infinite connected sum of closed $3$-manifolds, $\{\hat{M}_l\}_{l}$, where $\hat{M}_l$ is a closed and prime $3$-manifold.  For each $l$, we can decompose $M$ in the following way: $$M\cong \hat{M}_l\#M'_l,$$ where $M'_l$ is a non-compact $3$-manifold. 

It is sufficient to show that each $\hat{M}_l$ is homeomorphic to $\mathbb{S}^1\times \mathbb{S}^2$ or a spherical 3-manifold.

For each $\hat{M}_l$, there is an embedded $2$-sphere $\S$ satisfying that

\begin{itemize}
\item it bounds a compact $3$-manifold $M_l\subset M$;
\item $\i M_l$ is homeomorphic to $\hat{M}_l$ with a puncture. 
\end{itemize}

We assume that $\hat{M}_l$ is not homeomorphic to any of $\mathbb{S}^3$, $\mathbb{S}^1\times \mathbb{S}^2$ and $\mathbb{R}P^3$. It remains to prove that $\hat{M}_l$ is a spherical $3$-manifold. 

\vspace{1mm}

By Theorem \ref{mixed-existence}, there is an open $3$-manifold $M'_l\subset M$ so that 

 (1) $M'_l$ is homeomorphic to $\hat{M}_l$ with finitely many punctures; 

 (2) the closure of $(M'_{l}, g|_{M'_l})$ is a compact $3$-manifold with mixed boundary. 

We recall that any prime $3$-manifold is irreducible except $\mathbb{S}^1 \times \mathbb{S}^2$(see Proposition 1.4  of \cite{HA}). Then, we find that $\hat{M}_l$ is irreducible, which combines with Proposition \ref{ball} to get $\pi_2(\hat{M})$ is trivial.  Since $\hat{M}$ is the union of $M$ with some $3$-balls, we have that $\pi_2(M)$ is generated by the spheres in $\partial M$. 

Since  the scalar curvature  $\kappa\geq 1$, we use Theorem \ref{classify} to show that $M'_l$ is a spherical $3$-manifold with finitely many punctures. Namely, $\hat{M}_l$ is homeomorphic to a spherical $3$-manifold. 
\end{proof}
\begin{corollary} \label{mean-sphere} Let $(M, g)$ be an orientable  and connected $3$-manifold with scalar curvature $\kappa\geq 1$. If $\p M$ is a union of $2$-spheres and minimal for $g$, then $M$ is homeomorphic to an infinite connected sum of spherical $3$-manifolds, $3$-balls and some copies of $\mathbb{S}^1\times \mathbb{S}^2$. 
\end{corollary}
The proof is the same as the proof of Theorem \ref{A} but use Corollary \ref{mean-prime} and Remark \ref{mean-mixed} instead of Theorem \ref{C} and Theorem \ref{mixed-existence}.

\section{Proof of Theorem \ref{B}}
\subsection{Geometric version of Loop Lemma}
\begin{proposition}\label{loop} Let $(M, g)$ be a $3$-manifold with scalar curvature $\kappa\geq 1$ and with mean convex boundary. Assume that the boundary of $M$ is a union of  closed surfaces, $\{S_j\}_j$. If $\pi_1(S_{j_0})\rightarrow \pi_1(M)$ is not injective for some $j_0$, there is an embedded stable minimal disc $(D, \p D)\subset (M, S_{j_0})$ whose boundary is not homotopically trivial  in $S_{j_0}$. 
\end{proposition}

\begin{proof} Let $\sigma\subset S_{j_0}$  be a simple circle, which is in the kernel of $\pi_1(S_{j_0})\rightarrow \pi_1 (M)$. There is a constant $r_0>4\pi$ such that $\sigma$ is homotopically trivial in $B(S_{j_0}, r_0)$. 
 We consider
the set $\tilde{J}=\{j ~|~ S_j\cap B(S_{j_0}, r_0)\neq \emptyset\}$ and define that 
$$K:=B(S_{j_0}, r_0)\cup_{j\in \tilde{J}} B(S_j, \epsilon_j),$$
where $B(S_j, \epsilon_j)$ is a tubular neighborhood of $S_j$ and $\{B(S_j, \epsilon_j)\}_{j}$ are disjoint. 
Then, we have that 

(1) $\sigma$ is homotopically trivial in $K$; 

(2) $\p^*K=\p K\cap \i M$ is a disjoint union of some closed surfaces; 

(3) $d(S_{j_0}, \p^* K)\geq r_0>4\pi$. 

\noindent We notice that the set $\p^* K$ and $\p M$ are disjoint.

\vspace{1mm}

\noindent \textbf{Step 1}: Find a metric $g'$ on $K$ so that $(K, g')$ has mean convex boundary. 

\vspace{1mm}

Choose $\epsilon>0$ sufficiently small satisfying that 
\begin{itemize}[leftmargin=15pt]
\item $B(\p^*K, \epsilon)$ and $\p M$ are disjoint; 
\item $d(\p^{*}K', S_{j_0})>4\pi$,
\end{itemize} where $K':=K\setminus B(\p^*K, \epsilon)$ and $\p^*K'=\p K'\cap \i M$. 

We use Lemma \ref{mean-convex-deform} to find  a metric $g'$ on $K$ so that 
(1) $g'=g$ on $K'$;
(2) the mean curvature of $\p^* K$ is positive for $g'$. 

Since $(M, g)$ has mean convex boundary, the mean curvature of $\p M\cap K$ is still non-negative for $g'$. Therefore, $(K, g')$ has mean convex boundary. 

\vspace{1mm}

We use Theorem \ref{disc} to find an embedded disc $D\subset K$ with boundary $\sigma$. It is stable minimal for $g'$.
\vspace{1mm}

\noindent \textbf{Step 2}: Show that $D$ is stable minimal for $g$. 

It is sufficient to show that $D$ is a subset of $K'$. 

Suppose that $D$ is not contained in $K'$. Namely, $D\cap \p^* K'\neq \emptyset$.  Consider the component $D'$ of $D\cap K'$ connecting $S_{j}$ and $\p^*K'$. Since $g'=g$ on $K'$, $D'$ is stable minimal for $g$. Corollary \ref{dist1} gives that $d(S_{j_0}, \p^* K')\leq 4\pi$.  However, $d(\p^*K', S_{j_0})>4\pi$, a contradiction. 

We can conclude that $D$ is in $K'$. It is stable minimal for $g$. 
\end{proof}
\begin{remark}\label{bdy-top}
Let $M$ and $D$ be as in Proposition \ref{loop}. By Definition \ref{bdy-con-sum},  $M$ is homeomorphic to a boundary connected sum of the components of $M\setminus D$ (with itself). 
\end{remark}

We  inductively use Proposition \ref{loop} to get a $3$-manifold with boundary. For describing the boundary surfaces, we introduce a family of surfaces in $3$-manifolds. 

\begin{definition} Let $X$ be a $3$-manifold and $\Sigma\subset X$ be an immersed surface. The surface $\Sigma$ is called \emph{incompressible} if the induced map $\pi_1(\Sigma)\rightarrow \pi_1(M)$ is injective. 
\end{definition}

\begin{corollary}\label{cut-disc}Let $(M, g)$ and $\{S_j\}_{j\in J}$ be as in Proposition \ref{loop}.  For each $j$, there are finitely many embedded disks  $\{D^l_j\}^{L_j}_{l=1}$ so that 
\begin{itemize}[leftmargin=15pt]
\item each $(D^l_j, \partial D^l_j)\subset (M, S_j)$ is stable minimal; 
\item for each component $M_s$ of $M\setminus \amalg^{L_j}_{l=1} D^l_j$, each closed surface in $\p M_s$ coming from $S_j$ is incompressible in $M_s$.
\end{itemize}
\end{corollary}

\begin{corollary}\label{cut-disc1} Let $(M, g)$ and $\{S_j\}_{j\in J}$ be as in Proposition \ref{loop}. Then there are a family of stable minimal discs $\{D^l_j\}_{j\in J, 1\leq l\leq L_j}$ so that 
\begin{itemize}[leftmargin=15pt]
\item[(a)] each component of $M\setminus \amalg_{j\in J}(\amalg_{l}D^l_j)$ is a $3$-manifold with mean convex boundary; 
\item[(b)] the boundary of a component $M_s$ of $M\setminus \amalg_{j\in J}(\amalg_{l}D^l_j)$ is a union of some closed surfaces; 
\item[(c)] each component of the boundary of $M_s$ is incompressible in $M_s$. 
\end{itemize}
\end{corollary}

\begin{proof}We inductively apply Corollary \ref{cut-disc} to each $S_j$ and get finitely many  disjoint stable minimal discs $\{D_j^l\}^{L_j}_{l=1}$. The discs $\{D_j^l\}_{j\in J, 1\leq l\leq L_j}$ are also disjoint. 

\vspace{1mm}

It remains to show that the discs, $\{D^l_j\}_{j\in J, 1\leq l\leq L_j}$ are locally finite.

\vspace{1mm}

For any compact set $K\subset M$, consider   a compact set $B(K, 4\pi):=\{x~|~d(x, K)\leq 4\pi\}\subset M$. Since the family $\{S_j\}_{j\in J}$ is  locally finite, there are (at most) finitely many non-empty sets in $\{B(K, 4\pi)\cap S_j\}_{j\in J}$. 

Each $D^{l}_j$ is stable minimal  for $g$ and the scalar curvature $\kappa \geq 1$.  Proposition \ref{dist} shows that {if} $D^{l}_j\cap K\neq\emptyset$, we have that 
\begin{itemize}
\item  $d(K, \p D^{l}_j)\leq \frac{2\pi}{\sqrt{3}}$. 

\item  $d(K, S_j)\leq 2\pi$. Namely, $B(K, 4\pi)\cap S_j\neq \emptyset$. \end{itemize} 
Therefore, there are finitely many non-empty sets in $\{K\cap (\amalg^{L_j}_{l=1}D^l_j)\}_j$ (i.e. $\{D^l_j\}_{j\in J, 1\leq l\leq L_j}$ are locally finite). If not, there are infinitely many non-empty sets in $\{B(K, 4\pi)\cap S_j\}_{j\in J}$, a contradiction. 

\vspace{2mm}

The local finiteness of $\{D^l_j\}_{j\in J, 1\leq l\leq L_j}$ shows that each component of $M\setminus \amalg_j(\amalg_{l}D^l_j)$ is a $3$-manifold. Its boundary comes from minimal discs and $\{S_j\}_{j\in J}$. Thus, each component is a $3$-manifold with mean convex boundary. 

We could use Corollary \ref{cut-disc} to get that each component also satisfies (b) and (c). \end{proof}

\subsection{Incompressible surfaces and Positive scalar curvature} In this part, we talk about incompressible surfaces in a $3$-manifold with positive scalar curvature. 

\begin{lemma}\label{doubling} Let $(X, g)$ be a (non-compact) $3$-manifold with mean convex boundary and  scalar curvature $\kappa\geq 1$. If the boundary $\p X$ is a disjoint union of closed surfaces, $\{S_j\}_j$,  then the doubling of $X$ carries a complete metric of positive scalar curvature. 
\end{lemma}

The proof for the compact case is the same as [Theorem 1.1, Page 75] in \cite{Almeida}. Since the proof of  [Theorem 1.1, Page 75] in \cite{Almeida} only depends on the local computations,  it remains valid for the non-compact case. 

We use Lemma \ref{doubling} to get the following result. 

\begin{theorem}\label{bdy-sphere} Let $(X, g)$ be a $3$-manifold with scalar curvature $\kappa\geq 1$ and with mean convex boundary. If each component of $\p X$ is a closed  incompressible surface in $X$,  then each component of $\p X$ is homeomorphic to $\mathbb{S}^2$. 
\end{theorem}

\begin{proof} 
Suppose that there is an incompressible surface $S_j\subset \p X$ with the genus $g(S_j)\geq 1$.  
Since the map $\pi_1(\p X)\rightarrow \pi_1(X)$ is injective, we use the group theory's version of Van-Kampen's theorem (see  Theorem 11.60  of  \cite{Rot} on Page 386) to find that $S_j$ is also an incompressible surface in the doubling of $X$.

By Lemma \ref{doubling}, the doubling of $X$ has a complete metric with positive scalar curvature. 
However, it is well-known from Theorem 4  of \cite{SY} on Page 225 and Theorem 8.4 of \cite{GL} on Page 164 that if a $3$-manifold has an incompressible  closed surface whose genus is greater than zero, it has no complete metric with positive scalar curvature.  It leads to a contradiction. 

We can conclude that each component of $\p X$ is a $2$-sphere. 
\end{proof}

\subsection{Complete the proof of Theorem \ref{B}} We begin with a metric deformation  and then  use it to complete the proof of Theorem \ref{B}.

\begin{theorem} \label{metric-deformation}Let $(X, g)$ be a $3$-manifold with scalar curvature $\kappa\geq 1$ and with mean convex boundary.  If $\p X$ has a closed surface $S$, then for  any $\epsilon>0$, there is a metric $g'$ on $X$ so that 
\begin{itemize}[leftmargin=15pt]
\item the scalar curvature of $\kappa_{g'}\geq \frac{1}{2}$; 
\item the boundary $\p X$ is minimal for $g'$; 
\item $g'$ is equal to $g$ on $X\setminus B(S, \epsilon)$. \end{itemize}where $B(S, \epsilon)$ is the tubular neighborhood of $S$ with radius $\epsilon$. 
\end{theorem}

The metric $g'$ comes from a local metric deformation on $B(S, \epsilon)$. It was proved by Theorem 3.7 in \cite{Bar-Hanke} and also by  \cite{C-L}.

\begin{corollary}Let $(X, g)$ be a $3$-manifold with scalar curvature $\kappa\geq 1$ and with mean convex boundary. If $\p X$ is a union of closed surfaces $\{S_j\}_j$, there is a metric $g'$ on $X$ satisfying that 
\begin{itemize}
\item the scalar curvature of $g'$ is not less than $\frac{1}{2}$;
\item the boundary of $\p X$ is minimal for $g'$;
\end{itemize}
\end{corollary}

\begin{remark}\label{scaling} After scaling the metric $g'$, we have a metric $g''$ so that (1) the scalar curvature of $g''$ is bounded below by one (2) $\p X$ is minimal for $g''$. \end{remark}

\begin{proof} For each $j$, there is a positive constant $\epsilon_j$ so that $\{B(S_j, \epsilon_j)\}_{j}$ are disjoint. We apply Theorem \ref{metric-deformation} to each $S_j$ and deform the metrics on each $B(S_j, \epsilon_j)$. Thus, we could obtain a metric $g'$ such that (1) the scalar curvature of $g'$ is not less than $\frac{1}{2}$ (2) $\p X$ is minimal for $g'$. 
\end{proof}

We now complete the proof of   Theorem \ref{B}. 

\begin{proof} (Proof of Theorem \ref{B})
Let $(M, g)$ be an orientable $3$-manifold with mean convex boundary and with scalar curvature $\kappa\geq 1$. The boundary $\p M$ consists of some closed surfaces, $\{S_j\}_{j\in J}$.

We use Corollary \ref{cut-disc1} to get a family of stable minimal discs $\{D^{l}_j\}_{j\in J, 1\leq l\leq L_j}$ so that 
\begin{itemize}[leftmargin=15pt]
\item They cut $M$ into some connected $3$-manifolds, $\{M_s\}$. 
\item Each component of $\p M_s$ is incompressible in $M_s$ and a closed surface. 
\item $(M_s, g|_{M_s})$ has mean convex boundary and the scalar curvature $\kappa\geq 1$.  
\end{itemize}From Theorem \ref{bdy-sphere}, each component of $\p M_s$ is a $2$-sphere.

We use Remark \ref{scaling} to find  a metric on $M_s$ so that (1)  the scalar curvature is bounded below by one;  (2) $\p X_s$ is minimal for the new metric. By Corollary \ref{mean-sphere}, $M_s$ is an infinite connected sum of spherical $3$-manifolds, $3$-balls and some copies of $\mathbb{S}^1\times\mathbb{S}^2$. 

By Definition \ref{bdy-con-sum}, $M$ is homeomorphic to a boundary connected sum of $\{M_s\}_s$  (with themselves) (see Remark \ref{bdy-top}). We could  do the operation of the boundary connected sum on each $3$-ball (at most) finitely many times and obtain some handlebodies (see Remark \ref{handle-bdy-connected}). 

Therefore, $M$ is an infinite connected sum of spherical $3$-manifolds, handlebodies and some copies of $\mathbb{S}^1\times \mathbb{S}^2$. 
\end{proof} 

\section*{Appendix A: Proofs of Lemma \ref{cut} and Lemma \ref{curve}}
We will complete the proofs of Lemma \ref{cut} and Lemma \ref{curve}. 

\vspace{2mm}

\noindent \textbf{Lemma \ref{cut}} \emph{Let $M$ be an open $3$-manifold which is homotopically equivalent  to $\s^1$. Then a closed surface $\S\subset M$ cuts $M$ into a non-compact part and a compact part.}
\begin{proof} Suppose that $M\setminus \S$ is connected. Namely, there is a closed curve $\sigma$ which intersects $\S$ transversally at one point. The intersection number of $\S$ and $\sigma$ is $\pm 1$. 

However, since $H_2(M)$ is trivial ,  there is a compact connected $3$-manifold $M_0\subset M$ containing $\S$ and $\sigma$ satisfying  $[\S]=0$ in  $H_2(M_0)$. The Poincar\'e duality shows that  
\begin{equation*}
\begin{split}
 H_{1}(M_0, \p M_0)\times H_{2}(M_0)&\longrightarrow \mathbb{Z}\\
\sigma\quad\quad\times\quad\S\quad&\mapsto  (\sigma, \S).\\
\end{split}
\end{equation*}Therefore, the intersection number $(\sigma,\S)$ is equal to zero, which contradicts  the last paragraph.  We can conclude that $\S$ cuts $M$ into two components, $M_1$ and $M_2$.

\vspace{2mm}

We note that $\S$ is the boundary of $\overline{M}_i$ for $i=1,2$. If $M_i$ is non-compact,  $[\S]\neq 0$  in $H_2(\overline{M}_i)$. 

\vspace{2mm}

It remains  to show that one of components of $M\setminus \S$ is compact. 

\vspace{2mm}

The  Mayer-Vietoris sequence shows \[H_3(M)\rightarrow H_2(\S)\rightarrow H_2(\overline{M_1})\oplus H_2(\overline{M_2})\rightarrow H_2(M). \]
Since  $H_3(M)\cong H_2(M)\cong \{1\}$,
 we have that  $H_2(\overline{M}_1)\oplus H_2(\overline{M}_2)\cong\mathbb{Z}$. 
One of $H_2(\overline{M}_1)$ and $H_2(\overline{M}_2)$  is isomorphic to $\mathbb{Z}$ and generated by $[\S]$.

We may assume that $H_2(\overline{M}_1)$ is isomorphic to $\mathbb{Z}$ and generated by $[\S]$. 
Thus, $H_2(\overline{M}_2)$ is trivial. From the above remark, $\overline{M}_2$ is compact. 
Since $M$ is non-compact, $\overline{M}_1$ is non-compact.  
We finish the proof of the lemma. 
\end{proof}

\noindent\textbf{Lemma \ref{curve}} \emph{Let $M$ be an open $3$-manifold which is homotopically equivalent to $\mathbb{S}^1$ and $\gamma$ a circle generating $\pi_1(M)$. Assume that an embedded closed surface $\S$ bounds a compact set $M_{\S}\subset M$. If $\gamma$ is a subset of $M_{
\S}$, then there is a circle $\sigma\subset \S$ so that it is contractible in $M$ but non-contractible in $M\setminus \gamma$. }

\vspace{2mm}

Before proving Lemma \ref{curve}, we introduce two types of surgeries on a compact set. Let $M$ and $\gamma$ be assumed as in Lemma \ref{curve}. Assume that an embedded surface $\S\subset M$ bounds a compact set $M_{\S}\subset M$ and $\gamma\subset M_{\S}$. 

Consider an embedded disc $(D, \p D)\subset (M\setminus \gamma, \S)$ with $D\cap \S=\p D$, where $\p D$ is non-contractible in $\S$. 

\vspace{1mm}

\noindent \textbf{Type I}: If $D$ is a subset of $M_{\S}$,  we consider an open tubular neighborhood $N(D, \epsilon) \subset M_{\S}\setminus \gamma$ of $D$. We then have two cases: 

Case (1):  If $N(D, \epsilon )$ cuts $M_\S$ into two components, then we choose the compact set $M_{\S'}$ as the component of $M_{\S}\setminus N(\S, \epsilon)$ containing $\gamma$, where $\S'=\p M_{\S'}$ is a closed surface. 

Case (2): If $M_{\S}\setminus N(D, \epsilon)$ is connected, then we choose the compact set $M_{\S'}$ as $M_{\S}\setminus N(D, \epsilon)$, where $\S'=\p M_{\S'}$ is a closed surface. 

\vspace{2mm}

\noindent \textbf{Type II}: If $D$ is a subset of $\overline{M\setminus M_\S}$, then we consider an open tubular neighborhood $N(D, \epsilon)$ of $D$ in $\overline{M\setminus M_\S}$. We also have two cases. 

Case (1): If $\S\setminus \p D$ is connected, we choose $M_{\S'}$ as $M_\S\cup N(D, \epsilon)$, where $\S'=\p M_{\S'}$ is a closed surface. 

Case (2): If $\S\setminus \p D$ is disconnected, then the boundary of $M_\S\cup N(D, \epsilon)$ has two components. Choose the component $\S'$ satisfying that 
\begin{itemize}
\item it bounds a compact set $M_{\S'}\subset M$;
\item $M_\S$ is a subset of $M_{\S'}$. 
\end{itemize}Its existence is ensured by the proof of Corollary \ref{exhaustion}. 

\vspace{1mm}

\noindent \textbf{Remark}:  Let $\S'$ be a closed surface constructed from Type I or Type II surgery. We have that 

(1) $\gamma$ is a subset of $M_{\S'}$;

(2) In any case, $\S'$ is obtained by cutting $\S$ along $\p D$ and then gluing one component with some discs. We find that  $\S'$ is a union of $\S'\cap \S$ and some disjoint discs;

(3) As described in (2), since $\p D$ is not homotopically trivial in $\S$, we have  that the genus of  $ \S'$ is less than the genus of $\S$. 

\vspace{2mm}

We observe  that $g(\S) >0$. Otherwise, $\S$ is a $2$-sphere. Since $M$ is irreducible, then it bounds a $3$-ball $M_\S\subset M$. Thus, $\gamma$ is contractible in $M_\S\subset M$, which is a contradiction.  

\begin{proof} (Proof of Lemma \ref{curve}) We argue by induction on the genus $g(\S)$.

\vspace{2mm}

When $g(\S)=1$, then $\pi_1(\S)\rightarrow \pi_1(M\setminus \gamma)$ is injective. If not, there is an embedded disc $D\subset M\setminus \gamma$ with $D\cap \S=\p D$. The boundary $\p D\subset \S$ is non-contractible in $\S$. We use Type I or Type II surgery to get a closed surface $\S'$ satisfying that 
\begin{itemize}
\item it bounds a compact set $M_{\S'}\subset M$ and $\gamma\subset M_{\S'}$; 
\item $g(\S')<g(\S)$.
\end{itemize}Thus, $g(\S')$ is equal to zero (i.e. $\S'$ is a $2$-sphere and $M_{\S'}$ is a $3$-ball). We get that $\gamma$ is contractible in $M_{\S'}\subset M$, which leads to a contradiction. 

\vspace{2mm}

Since $\pi_1(M)\cong \mathbb{Z}$, then the map $\pi_1(\S)\rightarrow \pi_1(M)$ is not injective. We choose a circle $\sigma\subset \S$ so that $[\sigma]$ is a non-trivial element in $\text{Ker}(\pi_1(\S)\rightarrow \pi_1(M))$. Namely, $\sigma$ is contractible in $M$. 

Because $\pi_1(\S)\rightarrow \pi_1(M\setminus \gamma)$ is injective,   $\sigma$ is non-contractible in $M\setminus \gamma$. It is the required candidate in our assertion.

\vspace{1mm}

Suppose that it is true when $g(\S)\leq k$. 

\vspace{1mm}

When $g(\S)=k+1$, we can conclude that the map $\pi_1(\S)\rightarrow \pi_1(M\setminus \gamma)$ is injective. The reason is as follows: 

If not, there is an embedded disc $D\subset M\setminus \gamma$ with $D\cap \S=\p D$. The boundary $\p D\subset \S$ is non-contractible in $\S$. We use Type I or Type II surgery to get a closed surface $\S'$ satisfying that 
\begin{itemize}
\item[(a)] it bounds a compact set $M_{\S'}\subset M$ and $\gamma\subset M_{\S'}$; 
\item[(b)] $\S'$ is a union of $\S'\cap \S$ and some disjoint discs; 
\item[(c)] $g(\S')<g(\S)$.
\end{itemize}From (c), we apply the inductive hypothesis to $\S'$ and get a circle $\sigma_1 \subset \S'$. From (b), we may assume that $\sigma_1$ is a subset of $\Sigma\cap \S'\subset \S$. It is the required circle in our assertion. 

\vspace{2mm}

Since $\pi_1(M)\cong \mathbb{Z}$, then the map $\pi_1(\S)\rightarrow \pi_1(M)$ is not injective. We choose a circle $\sigma\subset \S$ so that $[\sigma]$ is a non-trivial element in $\text{Ker}(\pi_1(\S)\rightarrow \pi_1(M))$. Namely, $\sigma$ is contractible in $M$. 

Because $\pi_1(\S)\rightarrow \pi_1(M\setminus \gamma)$ is injective,  $\sigma$ is non-contractible in $M\setminus \gamma$. This is the required candidate in our assertion. It finishes the proof of Lemma \ref{curve}
\end{proof}

\bibliographystyle{alpha}
\bibliography{Uniform-positive}

\end{document}